\documentclass[11pt]{amsart}
\usepackage{amsmath,amsthm,amsfonts,amssymb,mathrsfs}
\usepackage{color}
\date{\today}

\usepackage{hyperref}
\input{xy}
 \xyoption{all}
 \xyoption{arc}

\newtheorem{theorem}{Theorem}[section]

\newtheorem{proposition}[theorem]{Proposition}
\newtheorem{corollary}[theorem]{Corollary}

\newtheorem{lemma}[theorem]{Lemma}

\theoremstyle{definition}
\newtheorem{definition}[theorem]{Definition}
\newtheorem{construction}[theorem]{Construction}
\newtheorem{example}[theorem]{Example}
\newtheorem{remark}[theorem]{Remark}

\begin{document}

\title[Extensions of semigroups by symmetric inverse semigroups of a~bounded finite rank]{Extensions of semigroups by symmetric inverse semigroups of a~bounded finite rank}
\author{Oleg~Gutik and Oleksandra Sobol}
\address{Faculty of Mechanics and Mathematics, Ivan Franko National University of Lviv, Universytetska 1, Lviv, 79000, Ukraine}
\email{oleg.gutik@lnu.edu.ua, ovgutik@yahoo.com, o.yu.sobol@gmail.com}

\keywords{Inverse semigroup, symmetric inverse semigroup of finite transformations, Green's relations, semigroup has a tight ideal series, semitopologica; semigroup, compact semigroup
}


\begin{abstract}
We study the semigroup extension $\mathscr{I}_\lambda^n(S)$ of a semigroup $S$ by symmetric inverse semigroups of a bounded finite rank.  We describe  idempotents and regular elements of the semigroups $\mathscr{I}_\lambda^n(S)$ and $\overline{\mathscr{I}_\lambda^n}(S)$ show that the semigroup $\mathscr{I}_\lambda^n(S)$ ($\overline{\mathscr{I}_\lambda^n}(S)$) is regular, orthodox, inverse or stable if and only if so is $S$. Green's relations are described on the semigroup $\mathscr{I}_\lambda^n(S)$ for an arbitrary monoid $S$. We introduce the conception of a semigroup with strongly tight ideal series, and  proved that for any infinite cardinal $\lambda$ and any positive integer $n$ the semigroup $\mathscr{I}_\lambda^n(S)$ has a  strongly tight ideal series provides so has $S$. At the finish we show that for every compact Hausdorff semitopological monoid $(S,\tau_S)$ there exists a unique its compact topological extension $\left(\mathscr{I}_\lambda^n(S),\tau_{\mathscr{I}}^\mathbf{c}\right)$ in the class of Haudorff semitopological semigroups.
\end{abstract}

\maketitle


\section{Introduction, motivation and main definitions}

In this paper we shall follow the terminology of~\cite{CP, Petrich1984}.

If $S$ is a~semigroup, then by $E(S)$ we denote the
subset of all idempotents of $S$. On  the set of idempotents
$E(S)$ there exists the natural partial order: $e\leqslant f$
\emph{if and only if} $ef=fe=e$.

A semigroup $S$ is called:
\begin{itemize}
  \item \emph{regular}, if for every $a\in S$ there exists an element $b$ in $S$ such that $a=aba$;
  \item \emph{orthodox}, if  $S$ is regular and $E(S)$ is a subsemigroup of $S$;
  \item \emph{inverse} if every $a$ in $S$ possesses an unique inverse, i.e. if there exists an unique element $a^{-1}$ in $S$ such that
  \begin{equation*}
    aa^{-1}a=a \qquad \mbox{and} \qquad a^{-1}aa^{-1}=a^{-1}.
\end{equation*}
\end{itemize}
It is obvious that every inverse semigroup is orthodox and every orthodox semigroup is regular.
A map which associates to any element of an inverse semigroup its
inverse is called the \emph{inversion}.

Let $\lambda$ be an arbitrary non-zero cardinal. A map $\alpha$ from a subset $D$ of $\lambda$ into $\lambda$ is called a \emph{partial transformation} of $X$. In this case
the set $D$ is called the \emph{domain} of $\alpha$ and is denoted
by $\operatorname{dom}\alpha$. Also, the set $\{ x\in X\colon
y\alpha=x \mbox{ for some } y\in Y\}$ is called the \emph{range}
of $\alpha$ and is denoted by $\operatorname{ran}\alpha$. The
cardinality of $\operatorname{ran}\alpha$ is called the
\emph{rank} of $\alpha$ and denoted by
$\operatorname{rank}\alpha$. For convenience we denote by
$\varnothing$ the empty transformation, that is a partial mapping
with
$\operatorname{dom}\varnothing=\operatorname{ran}\varnothing=\varnothing$.

Let $\mathscr{I}_\lambda$ denote the set of all partial one-to-one transformations of $\lambda$ together with the following semigroup operation:
\begin{equation*}
    x(\alpha\beta)=(x\alpha)\beta \quad \mbox{if} \quad
    x\in\operatorname{dom}(\alpha\beta)=\{
    y\in\operatorname{dom}\alpha\colon
    y\alpha\in\operatorname{dom}\beta\}, \qquad \mbox{for} \quad
    \alpha,\beta\in\mathscr{I}_\lambda.
\end{equation*}
The semigroup $\mathscr{I}_\lambda$ is called the \emph{symmetric
inverse semigroup} over the cardinal $\lambda$~(see \cite{CP}). The symmetric
inverse semigroup was introduced by V.~V.~Wagner~\cite{Wagner1952}
and it plays a major role in the theory of semigroups.

Put
\begin{equation*}
\mathscr{I}_\lambda^\infty=\{ \alpha\in\mathscr{I}_\lambda\colon
\operatorname{rank}\alpha \;\mbox{ is finite}\}  \quad \mbox{ and }
\quad \mathscr{I}_\lambda^n=\{ \alpha\in\mathscr{I}_\lambda\colon
\operatorname{rank}\alpha\leqslant n\},
\end{equation*}
for $n=1,2,3,\ldots$. Obviously, $\mathscr{I}_\lambda^\infty$ and
$\mathscr{I}_\lambda^n$ ($n=1,2,3,\ldots$) are inverse semigroups,
$\mathscr{I}_\lambda^\infty$ is an ideal of $\mathscr{I}_\lambda$, and
$\mathscr{I}_\lambda^n$ is an ideal of
$\mathscr{I}_\lambda^\infty$, for each $n=1,2,3,\ldots$. Further,
we shall call the semigroup $\mathscr{I}_\lambda^\infty$ the
\emph{symmetric inverse semigroup of finite transformations} and
$\mathscr{I}_\lambda^n$ the \emph{symmetric inverse semigroup of
finite transformations of the rank $\leqslant n$}. The elements of
semigroups $\mathscr{I}_\lambda^\infty$ and
$\mathscr{I}_\lambda^n$ are called \emph{finite one-to-one
transformations} (\emph{partial bijections}) of the cardinal $\lambda$. By
\begin{equation*}
\left(%
\begin{array}{ccc}
  x_1 & \cdots & x_n \\
  y_1 & \cdots & y_n \\
\end{array}%
\right)
\end{equation*}
we denote a partial one-to-one transformation which maps $x_1$ onto $y_1$, $\ldots$, $x_n$ onto $y_n$, and by $0$ the empty transformation. Obviously, in such case we have $x_i\neq x_j$ and $y_i\neq y_j$ for $i\neq j$ ($i,j=1,\ldots,n$). The empty partial map $\varnothing\colon \lambda\rightharpoonup\lambda$ we denote by $0$. It is obvious that $0$ is zero of the semigroup $\mathscr{I}_\lambda^n$.

Let $\lambda$ be a non-zero cardinal. On the set
 $
 B_{\lambda}=(\lambda\times\lambda)\cup\{ 0\}
 $,
where $0\notin\lambda\times\lambda$, we define the semigroup
operation ``$\, \cdot\, $'' as follows
\begin{equation*}
(a, b)\cdot(c, d)=
\left\{
  \begin{array}{cl}
    (a, d), & \hbox{ if~ } b=c;\\
    0, & \hbox{ if~ } b\neq c,
  \end{array}
\right.
\end{equation*}
and $(a, b)\cdot 0=0\cdot(a, b)=0\cdot 0=0$ for $a,b,c,d\in
\lambda$. The semigroup $B_{\lambda}$ is called the
\emph{semigroup of $\lambda\times\lambda$-matrix units}~(see
\cite{CP}). Obviously, for any cardinal $\lambda>0$, the semigroup
of $\lambda\times\lambda$-matrix units $B_{\lambda}$ is isomorphic
to $\mathscr{I}_\lambda^1$.

Let $S$ be a semigroup with zero and $\lambda$ be a non-zero cardinal. We define the semigroup operation on the set $B_{\lambda}(S)=(\lambda\times S\times
\lambda)\cup\{ 0\}$ as follows:
\begin{equation*}
 (\alpha,a,\beta)\cdot(\gamma, b, \delta)=
\left\{
  \begin{array}{cl}
    (\alpha, ab, \delta), & \hbox{ if } \beta=\gamma;\\
    0, & \hbox{ if } \beta\ne \gamma,
  \end{array}
\right.
\end{equation*}
and $(\alpha, a, \beta)\cdot 0=0\cdot(\alpha, a, \beta)=0\cdot 0=0$, for all $\alpha, \beta, \gamma, \delta\in \lambda$ and $a, b\in S$. If $S=S^1$ then the semigroup $B_\lambda(S)$ is called
the {\it Brandt $\lambda$-extension of the semigroup}
$S$~\cite{Gutik1999, GutikPavlyk2001}. Obviously, if $S$ has zero
then ${\mathcal J}=\{ 0\}\cup\{(\alpha, 0_S, \beta)\colon 0_S$ is
the zero of $S\}$ is an ideal of $B_\lambda(S)$. We put
$B^0_\lambda(S)=B_\lambda(S)/{\mathcal J}$ and the semigroup
$B^0_\lambda(S)$ is called the {\it Brandt $\lambda^0$-extension
of the semigroup $S$ with zero}~\cite{GutikPavlyk2006}.

A {\it topological} ({\it inverse}) {\it semigroup} is a Hausdorff topological space together with a continuous semigroup operation (and an~inversion, respectively). A {\it semitopological semigroup} is a Hausdorff topological space together with a separately continuous semigroup operation.

The Brandt $\lambda$-extension $B_{\lambda}(S)$ (or the Brandt $\lambda^0$-extension $B^0_{\lambda}(S)$) of a semigroup $S$ we may to consider as a some semigroup extension of the semigroup $S$ by the semigroup $\lambda\times\lambda$-matrix units $B_\lambda$. An analogue of so extension gives the following definition.

\section{The construction of of the semigroup extension $\mathscr{I}_\lambda^n(S)$}

In this paper using the semigroup $\mathscr{I}_\lambda^n$ we propose the following semigroup extension.

\begin{construction}\label{construction-1.1}
Let $S$ be a semigroup, $\lambda$ be a non-zero cardinal, $n$ and $k$ be a positive integers such that  $k\leqslant n\leqslant\lambda$. We identify every element $\alpha\in \mathscr{I}_\lambda^n$ with its graph $\operatorname{\textsf{Gr}}(\alpha)\subset\lambda\times\lambda$ and put
\begin{equation*}
    \mathscr{I}_\lambda^n(S)=\left\{\alpha_S\colon\operatorname{\textsf{Gr}}(\alpha) \rightarrow S \colon \alpha\in\mathscr{I}_\lambda^n\right\}
\end{equation*}
and every map from the empty map $0$ into $S$ we identify with the empty map $\varnothing\colon \lambda\times \lambda\rightharpoonup S$ and denote its by $0$. An arbitrary element  $0\neq\operatorname{rank}\alpha \leqslant n$ we denote by
\begin{equation*}
{\small{\left(%
\begin{array}{ccc}
  x_1 & \cdots & x_k \\
  s_1 & \cdots & s_k \\
  y_1 & \cdots & y_k
\end{array}%
\right)}},
\end{equation*}
where
$\alpha=
\left(%
\begin{array}{ccc}
  x_1 & \cdots & x_k \\
  y_1 & \cdots & y_k
\end{array}%
\right)$, and $\left((x_1,y_1)\right)\alpha=s_1$,  $\ldots,$ $\left((x_k,y_k)\right)\alpha=s_k$. Also for $\alpha_S\in\mathscr{I}_\lambda^n(S)$ such that
\begin{equation*}
\alpha_S=
{\small{\left(%
\begin{array}{ccc}
  x_1 & \cdots & x_k \\
  s_1 & \cdots & s_k \\
  y_1 & \cdots & y_k
\end{array}%
\right)}}
\end{equation*}
we denote $\mathop{\textsf{\textbf{d}}}(\alpha_S)=\{x_1,\ldots,x_k\}$ and $\mathop{\textsf{\textbf{r}}}(\alpha_S)=\{y_1,\ldots,y_k\}$.

Now, we define a binary operation ``$\cdot$'' on the set $\mathscr{I}_\lambda^n(S)$ in the following way:
\begin{itemize}
  \item[$(i)$] $\alpha_S\cdot 0=0\cdot\alpha_S=0\cdot 0=0$ for every  $\alpha_S\in\mathscr{I}_\lambda^n(S)$;
  \item[$(ii)$] if $\alpha\cdot \beta=0$ in $\mathscr{I}_\lambda^n$ then $\alpha_S\cdot \beta_S=0$ for any $\alpha_S,\beta_S \in \mathscr{I}_\lambda^n(S)$;
  \item[$(iii)$] if $\alpha_S=
  {\small{\left(%
  \begin{array}{ccc}
   a_1 & \cdots & a_i \\
   s_1 & \cdots & s_i \\
   b_1 & \cdots & b_i
  \end{array}%
  \right)}}
  $, $\beta_S=
  {\small{\left(%
  \begin{array}{ccc}
   c_1 & \cdots & c_j \\
   t_1 & \cdots & t_j \\
   d_1 & \cdots & d_j
  \end{array}%
  \right)}}
  $ and
  \begin{equation*}
  \alpha\cdot \beta=\left(%
  \begin{array}{ccc}
   a_1 & \cdots & a_i \\
   b_1 & \cdots & b_i
  \end{array}%
  \right)
  \cdot \left(%
  \begin{array}{ccc}
   c_1 & \cdots & c_j \\
   d_1 & \cdots & d_j
  \end{array}%
  \right)=
  \left(%
  \begin{array}{ccc}
   a_{i_1} & \cdots & a_{i_m} \\
   d_{j_1} & \cdots & d_{j_m}
  \end{array}%
  \right)\neq 0 \quad \hbox{~in~} \mathscr{I}_\lambda^n,
  \end{equation*}
  then $\alpha_S\cdot \beta_S={\small{\left(%
  \begin{array}{cccc}
   a_{i_1}        & \cdots & a_{i_m} \\
   s_{i_1}t_{j_1} & \cdots & s_{i_m}t_{j_m} \\
   d_{j_1}        & \cdots & d_{j_m}
  \end{array}%
  \right)}}$.
  \end{itemize}

Simple verifications show that so defines binary operation on $\mathscr{I}_\lambda^n(S)$ is associative and hence $\mathscr{I}_\lambda^n(S)$ is a semigroup. It is obvious that $\mathscr{I}_\lambda^1(S)$ is isomorphic to the Brandt $\lambda$-extension $B_\lambda(S)$ of the semigroup $S$.
\end{construction}

We remark that if the semigroup $S$ contains a zero $0_S$ then
\begin{equation*}
{\mathcal J}_{0}=\{ 0\}\cup\left\{\alpha_S=
  {\small{\left(%
  \begin{array}{ccc}
   a_1 & \cdots & a_i \\
   0_S & \cdots & 0_S \\
   b_1 & \cdots & b_i
  \end{array}%
  \right)}}\colon 0_S \hbox{ is
the zero of } S\right\}
\end{equation*}
is an ideal of $\mathscr{I}_\lambda^n(S)$.

Also, we define a binary relation $\equiv_0$ on the semigroup $\mathscr{I}_\lambda^n(S)$ in the following way. For $\alpha_S,\beta_S\in\mathscr{I}_\lambda^n(S)$ we put $\alpha_S\equiv_0\beta_S$ if and only if at least one of the following conditions holds:
\begin{itemize}
  \item[$(1)$] $\alpha_S=\beta_S$;
  \item[$(2)$] $\alpha_S,\beta_S\in{\mathcal J}_{0}$;
  \item[$(3)$] $\alpha_S,\beta_S\notin{\mathcal J}_{0}$ and each of conditions
  \begin{itemize}
    \item[$(i)$] $(x,y)\alpha_S$ is determined and $(x,y)\alpha_S\neq 0_S$; \; and
    \item[$(ii)$] $(x,y)\beta_S$ is determined and $(x,y)\beta_S\neq 0_S$
  \end{itemize}
    implies the equality $(x,y)\alpha_S=(x,y)\beta_S$.
\end{itemize}
It is obvious that $\equiv_0$ is an equivalence relation on the semigroup $\mathscr{I}_\lambda^n(S)$.

The following proposition prove by usual verifications.

\begin{proposition}\label{proposition-1.1}
The relation $\equiv_0$ is a congruence on the semigroup $\mathscr{I}_\lambda^n(S)$.
\end{proposition}

We define $\overline{\mathscr{I}_\lambda^n}(S)= \mathscr{I}_\lambda^n(S)/_{\equiv_0}$.

In this paper we study algebraic properties of the semigroups $\mathscr{I}_\lambda^n(S)$ and $\overline{\mathscr{I}_\lambda^n}(S)$. We describe  idempotents and regular elements of the semigroups $\mathscr{I}_\lambda^n(S)$ and $\overline{\mathscr{I}_\lambda^n}(S)$ show that the semigroup $\mathscr{I}_\lambda^n(S)$ ($\overline{\mathscr{I}_\lambda^n}(S)$) is regular, orthodox, inverse or stable if and only if so is $S$. Green's relations are described in the semigroup $\mathscr{I}_\lambda^n(S)$ for an arbitrary monoid $S$. We introduce the conception of a semigroup with strongly tight ideal series, and  proved that for any infinite cardinal $\lambda$ and any positive integer $n$ the semigroup $\mathscr{I}_\lambda^n(S)$ has a  strongly tight ideal series provides so has $S$. At the finish we show that for every compact Hausdorff semitopological monoid $(S,\tau_S)$ there exists a unique its compact topological extension $\left(\mathscr{I}_\lambda^n(S),\tau_{\mathscr{I}}^\mathbf{c}\right)$ in the class of Haudorff semitopological semigroups.


\section{Algebraic properties of the semigroup extensions $\mathscr{I}_\lambda^n(S)$ and $\overline{\mathscr{I}_\lambda^n}(S)$}

The following proposition describes the subset of idempotents of the semigroup $\mathscr{I}_\lambda^n(S)$.

\begin{proposition}\label{proposition-2.1}
For every positive integer $i\leqslant n$ a non-zero element
$\alpha_S=
  {\small{\left(%
  \begin{array}{ccc}
   a_1 & \cdots & a_i \\
   s_1 & \cdots & s_i \\
   b_1 & \cdots & b_i
  \end{array}%
  \right)}}
  $
of the semigroup $\mathscr{I}_\lambda^n(S)$ is an idempotent if and only if $a_1=b_1, \ldots, a_i=b_i$ and $s_1,\ldots,s_i\in E(S)$.
\end{proposition}

\begin{proof}
The implication $(\Leftarrow)$ is trivial.

$(\Rightarrow)$ Suppose that $\alpha_S\cdot \alpha_S=\alpha_S$. Then the definition of the semigroup $\mathscr{I}_\lambda^n(S)$ implies that the symbols $a_1,\ldots, a_i$ are distinct. Similar we get that the symbols $b_1,\ldots, b_i$ are distinct, too. The above arguments and the equality $\alpha_S\cdot \alpha_S=\alpha_S$ imply that $\left\{a_1,\ldots, a_i\right\}=\left\{b_1,\ldots, b_i\right\}$. Assume that $a_k\neq b_k=a_l$ for some integers $k,l\in\{1,\ldots,i\}$, $k\neq l$. Then we have that $a_l\neq b_l\neq b_k$, which contradicts the equality $\alpha_S\cdot \alpha_S=\alpha_S$. The obtained contradiction implies the equalities $a_1=b_1, \ldots, a_i=b_i$. Now, we get that
\begin{equation*}
\alpha_S\cdot \alpha_S= {\small{\left(\!\!
  \begin{array}{ccc}
   a_1 & \cdots & a_i \\
   s_1 & \cdots & s_i \\
   a_1 & \cdots & a_i
  \end{array}\!\!
  \right)}}
{\cdot}
{\small{\left(\!\!
  \begin{array}{ccc}
   a_1 & \cdots & a_i \\
   s_1 & \cdots & s_i \\
   a_1 & \cdots & a_i
  \end{array}\!\!
  \right)}}
={\small{\left(\!\!
  \begin{array}{ccc}
   a_1    & \cdots & a_i    \\
   s_1s_1 & \cdots & s_is_i \\
   a_1    & \cdots & a_i
  \end{array}\!\!
  \right)}}
= {\small{\left(\!\!
  \begin{array}{ccc}
   a_1 & \cdots & a_i \\
   s_1 & \cdots & s_i \\
   a_1 & \cdots & a_i
  \end{array}\!\!
  \right)}}=
\alpha_S,
\end{equation*}
and hence $s_1s_1=s_1, \ldots, s_is_i=s_i$. This completes the proof if the proposition.
\end{proof}


\begin{proposition}\label{proposition-2.2}
For every positive integer $i\leqslant n$ a non-zero element
$\alpha_S=
  {\small{\left(%
  \begin{array}{ccc}
   a_1 & \cdots & a_i \\
   s_1 & \cdots & s_i \\
   b_1 & \cdots & b_i
  \end{array}%
  \right)}}
  $
of the semigroup $\mathscr{I}_\lambda^n(S)$ is regular if and only if so are  $s_1,\ldots,s_i$ in $S$.
\end{proposition}

\begin{proof}
The implication $(\Leftarrow)$ is trivial. Indeed, $\alpha_S=\alpha_S\beta_S\alpha_S$ for $\beta_S=
  {\small{\left(%
  \begin{array}{ccc}
   b_1 & \cdots & b_i \\
   t_1 & \cdots & t_i \\
   a_1 & \cdots & a_i
  \end{array}%
  \right)}}
$, where elements $t_1,\ldots,t_i$ of the semigroup $S$ such that $s_1=s_1t_1s_1$, $\ldots,$  $s_i=s_it_is_i$.

$(\Rightarrow)$ Suppose that $\alpha_S$ is a regular element of the semigroup $\mathscr{I}_\lambda^n(S)$. Then there exists an element $\beta_S=
  {\small{\left(%
  \begin{array}{ccc}
   c_1 & \cdots & c_k \\
   t_1 & \cdots & t_k \\
   d_1 & \cdots & d_k
  \end{array}%
  \right)}}
$ of the semigroup $\mathscr{I}_\lambda^n(S)$, $0<k\leqslant n$, such that $\alpha_S=\alpha_S\cdot\beta_S\cdot\alpha_S$. Now, this implies that $\{b_1, \ldots, b_i\}\subseteq\{c_1, \ldots, c_k\}$ and hence $k\geqslant i$. Without loss of generality we may assume that $b_1=c_1, \ldots, b_i=c_i$. Then the equality $\alpha_S=\alpha_S\cdot\beta_S\cdot\alpha_S$ and the semigroup operation of $\mathscr{I}_\lambda^n(S)$ imply that $d_1=a_1, \ldots, d_i=a_i$ and hence we have that
\begin{equation*}
\begin{split}
  \alpha_S= & \;\alpha_S\cdot\beta_S\cdot\alpha_S=
  {\small{\left(%
  \begin{array}{ccc}
   a_1 & \cdots & a_i \\
   s_1 & \cdots & s_i \\
   b_1 & \cdots & b_i
  \end{array}%
  \right)}}\cdot
  {\small{\left(%
  \begin{array}{ccc}
   c_1 & \cdots & c_k \\
   t_1 & \cdots & t_k \\
   d_1 & \cdots & d_k
  \end{array}%
  \right)}}
    \cdot
  {\small{\left(%
  \begin{array}{ccc}
   a_1 & \cdots & a_i \\
   s_1 & \cdots & s_i \\
   b_1 & \cdots & b_i
  \end{array}%
  \right)}}=  \\
   = &
  {\small{\left(%
  \begin{array}{ccc}
   a_1 & \cdots & a_i \\
   s_1 & \cdots & s_i \\
   b_1 & \cdots & b_i
  \end{array}%
  \right)}}\cdot
  {\small{\left(%
  \begin{array}{cccccc}
   b_1 & \cdots & b_i & c_{i+1} & \cdots & c_k \\
   t_1 & \cdots & t_i & t_{i+1} & \cdots & t_k\\
   a_1 & \cdots & a_i & d_{i+1} & \cdots & d_k
  \end{array}%
  \right)}}
    \cdot
  {\small{\left(%
  \begin{array}{ccc}
   a_1 & \cdots & a_i \\
   s_1 & \cdots & s_i \\
   b_1 & \cdots & b_i
  \end{array}%
  \right)}}=  \\
  = & {\small{\left(%
  \begin{array}{ccc}
   a_1       & \cdots & a_i \\
   s_1t_1s_1 & \cdots & s_it_is_i \\
   b_1       & \cdots & b_i
  \end{array}%
  \right)}}=
  {\small{\left(%
  \begin{array}{ccc}
   a_1 & \cdots & a_i \\
   s_1 & \cdots & s_i \\
   b_1 & \cdots & b_i
  \end{array}%
  \right)}}.
\end{split}
\end{equation*}
This implies that the following equalities $s_1=s_1t_1s_1$, $\ldots,$ $s_i=s_it_is_i$ hold in $S$, which completes the proof of our proposition.
\end{proof}

Two elements $a$ and $b$ of a semigroup $S$ are said to be \emph{inverses} of each other if
\begin{equation*}
aba = a \qquad \hbox{and} \qquad bab = b.
\end{equation*}

The definition of the semigroup operation in $\mathscr{I}_\lambda^n(S)$ implies the following proposition.

\begin{proposition}\label{proposition-2.3}
Let $\lambda$ be a non-zero cardinal, $n$ and $i$ be any positive integers such that $i\leqslant n\leqslant\lambda$. Let $S$ be a semigroup and $a_1,\ldots,a_i,b_1,\ldots,b_i\in\lambda$. If the elements $s_1$ and $t_1$, $\ldots ,$ $s_i$ and $t_i$ are pairwise inverses of each other in $S$ then the elements
$
    {\small{\left(%
  \begin{array}{ccc}
   a_1 & \cdots & a_i \\
   s_1 & \cdots & s_i \\
   b_1 & \cdots & b_i
  \end{array}%
  \right)}}
$ and
$    {\small{\left(%
  \begin{array}{ccc}
   b_1 & \cdots & b_i\\
   t_1 & \cdots & t_i\\
   a_1 & \cdots & a_i \\
  \end{array}%
  \right)}}
$
are pairwise inverses of each other in the semigroup $\mathscr{I}_\lambda^n(S)$.
\end{proposition}

For arbitrary semigroup $S$, every positive integer $i\leqslant n$, any collection non-empty subsets $A_1,\ldots, A_i$ of $S$,  and any two collections of distinct elements $\{a_1,\ldots,a_i\}$ and $\{b_1,\ldots,b_i\}$ of the cardinal $\lambda$ we denote a subset
\begin{equation*}
[A_1,\ldots, A_i]^{(a_1,\ldots,a_i)}_{(b_1,\ldots,b_i)}=\left\{
  {\small{\left(%
  \begin{array}{ccc}
   a_1 & \cdots & a_i \\
   s_1 & \cdots & s_i \\
   b_1 & \cdots & b_i
  \end{array}%
  \right)}}\colon s_1\in A_1, \ldots, s_i\in A_i\right\}
  \end{equation*}
of $\mathscr{I}_\lambda^n(S)$. I the case when $A_1=\ldots=A_i=A$ in $S$ we put $[A]^{(a_1,\ldots,a_i)}_{(b_1,\ldots,b_i)}=[A_1,\ldots, A_i]^{(a_1,\ldots,a_i)}_{(b_1,\ldots,b_i)}$. It is obvious that for every subset $A$ of the semigroup $S$ and any permutation $\sigma\colon\{1,\ldots,i\}\rightarrow\{1,\ldots,i\}$ we have that
\begin{equation*}
[A]^{(a_{(1)\sigma},\ldots,a_{(i)\sigma})}_{(b_{(1)\sigma},\ldots,b_{(i)\sigma})}= [A]^{(a_1,\ldots,a_i)}_{(b_1,\ldots,b_i)}.
\end{equation*}

\begin{proposition}\label{proposition-2.4}
Let $\lambda$ be a non-zero cardinal and $n$ be any positive integer $\leqslant\lambda$. Then for arbitrary semigroup $S$, every positive integer $i\leqslant n$ and any collection of distinct elements $\{a_1,\ldots,a_i\}$ of $\lambda$ the direct power $S^i$ is isomorphic to a subsemigroup $S^{(a_1,\ldots,a_i)}_{(a_1,\ldots,a_i)}$ of $\mathscr{I}_\lambda^n(S)$.
\end{proposition}

\begin{proof}
The semigroup operation of $\mathscr{I}_\lambda^n(S)$ implies that $S^{a_1,\ldots,a_i}_{a_1,\ldots,a_i}$ is a subsemigroup of $\mathscr{I}_\lambda^n(S)$ for any collection of distinct elements $\{a_1,\ldots,a_i\}$ of $\lambda$. An isomorphism $\mathfrak{h}\colon S^i\rightarrow S^{(a_1,\ldots,a_i)}_{(a_1,\ldots,a_i)}$ we define by the formula $(s_1,\ldots,s_i)\mathfrak{h}=
  {\small{\left(%
  \begin{array}{ccc}
   a_1 & \cdots & a_i \\
   s_1 & \cdots & s_i \\
   a_1 & \cdots & a_i
  \end{array}%
  \right)}}$.
\end{proof}



\begin{proposition}\label{proposition-2.5}
For every semigroup $S$, any non-zero cardinal $\lambda$ and any positive integer $n\leqslant\lambda$ the following statements hold:
\begin{itemize}
  \item[$(i)$] $\mathscr{I}_\lambda^n(S)$ is regular if and only if so is $S$;
  \item[$(ii)$] $\mathscr{I}_\lambda^n(S)$ is orthodox if and only if so is $S$;
  \item[$(iii)$] $\mathscr{I}_\lambda^n(S)$ is inverse if and only if so is $S$.
\end{itemize}
\end{proposition}

\begin{proof}
Statement $(i)$ follows from Proposition~\ref{proposition-2.2}.

$(ii)$ $(\Leftarrow)$ Suppose that $S$ is an orthodox semigroup. Then statement $(i)$ implies that the semigroup $\mathscr{I}_\lambda^n(S)$ is regular. By Proposition~\ref{proposition-2.1} every non-zero idempotent of the semigroup $\mathscr{I}_\lambda^n(S)$ has the form
$
{\small{\left(%
  \begin{array}{ccc}
   a_1 & \cdots & a_i \\
   e_1 & \cdots & e_i \\
   a_1 & \cdots & a_i
  \end{array}%
  \right)}}
$,
where $0<i\leqslant n$ and $e_1, \ldots, e_i$ are idempotents of $S$. This implies that the product of two idempotents of $\mathscr{I}_\lambda^n(S)$ is again an idempotent, and hence the semigroup $\mathscr{I}_\lambda^n(S)$ is orthodox.

$(\Rightarrow)$ Suppose that $\mathscr{I}_\lambda^n(S)$ is an orthodox semigroup. By Proposition~\ref{proposition-2.4}, $S^{(a)}_{(a)}$ is a subsemigroup of $\mathscr{I}_\lambda^n(S)$ for every $a\in\lambda$ and hence $S^{(a)}_{(a)}$ is orthodox. Then Proposition~\ref{proposition-2.4} implies the semigroup $S$ is orthodox, too.

$(iii)$ $(\Leftarrow)$ Suppose that $S$ is an inverse semigroup. By statement $(i)$ the semigroup $\mathscr{I}_\lambda^n(S)$ is regular. Then using Proposition~\ref{proposition-2.1} we get that idempotents commute in $\mathscr{I}_\lambda^n(S)$ and hence by Theorem~1.17 of \cite{CP}, $\mathscr{I}_\lambda^n(S)$ is an inverse semigroup.

$(\Rightarrow)$ Suppose that $\mathscr{I}_\lambda^n(S)$ is an inverse semigroup. By Proposition~\ref{proposition-2.4}, $S^{(a)}_{(a)}$ is a subsemigroup of $\mathscr{I}_\lambda^n(S)$ for every $a\in\lambda$, and by Proposition~\ref{proposition-2.3} it is an inverse subsemigroup. Hence by Proposition~\ref{proposition-2.4}, $S$ is an inverse semigroup.
\end{proof}

Since any homomorphic image of a regular (resp., orthodox, inverse) semigroup is a   regular (resp., orthodox, inverse) semigroup (see \cite[Section~7.4]{CP} and \cite[Lemma~2.2]{Meakin1971}),
Proposition~\ref{proposition-2.5} implies the following corollary.

\begin{corollary}\label{corollary-2.6}
For every semigroup $S$, any non-zero cardinal $\lambda$ and any positive integer $n\leqslant\lambda$ the following statements hold:
\begin{itemize}
  \item[$(i)$] $\overline{\mathscr{I}_\lambda^n}(S)$ is regular if and only if so is $S$;
  \item[$(ii)$] $\overline{\mathscr{I}_\lambda^n}(S)$ is orthodox if and only if so is $S$;
  \item[$(iii)$] $\overline{\mathscr{I}_\lambda^n}(S)$ is inverse if and only if so is $S$.
\end{itemize}
\end{corollary}

If $S$ is a semigroup, then we shall denote by $\mathscr{R}$,
$\mathscr{L}$, $\mathscr{J}$, $\mathscr{D}$ and $\mathscr{H}$ the
Green relations on $S$ (see \cite{GreenJ1951} or \cite[Section~2.1]{CP}):
\begin{center}
\begin{tabular}{rcl}
  $a\mathscr{R}b$ & if and only if & $aS^1=bS^1$; \\
  $a\mathscr{L}b$ & if and only if & $S^1a=S^1b$; \\
  $a\mathscr{J}b$ & if and only if & $S^1aS^1=S^1bS^1$; \\
    & $\mathscr{D}=\mathscr{L}{\circ}\mathscr{R}=\mathscr{R}{\circ}\mathscr{L}$; &\\
    & $\mathscr{H}=\mathscr{L}\cap\mathscr{R}$. &\\
\end{tabular}
\end{center}

\begin{remark}\label{remark-2.7}
It is obvious that for non-zero elements $\alpha_S=
  {\small{\left(\!\!
  \begin{array}{ccc}
   a_1 & \cdots & a_i \\
   s_1 & \cdots & s_i \\
   b_1 & \cdots & b_i
  \end{array}\!\!
  \right)}}
  $
and $\beta_S=
  {\small{\left(\!\!
  \begin{array}{ccc}
   c_1 & \cdots & c_k \\
   t_1 & \cdots & t_k \\
   d_1 & \cdots & d_k
  \end{array}\!\!
  \right)}}
$ of the semigroup $\mathscr{I}_\lambda^n(S)$ any of conditions $\alpha_S\mathscr{R}\beta_S$, $\alpha_S\mathscr{L}\beta_S$, $\alpha_S\mathscr{D}\beta_S$, $\alpha_S\mathscr{J}\beta_S$, or $\alpha_S\mathscr{H}\beta_S$ implies the equality $i=k$.
\end{remark}

The following proposition describes the Green relations on the semigroup $\mathscr{I}_\lambda^n(S)$.

\begin{proposition}\label{proposition-2.8}
Let $S$ be a monoid, $\lambda$ be any non-zero cardinal and $n\leqslant\lambda$. Let $\alpha_S=
  {\small{\left(\!\!
  \begin{array}{ccc}
   a_1 & \cdots & a_i \\
   s_1 & \cdots & s_i \\
   b_1 & \cdots & b_i
  \end{array}\!\!
  \right)}}
  $
and $\beta_S=
  {\small{\left(\!\!
  \begin{array}{ccc}
   c_1 & \cdots & c_i \\
   t_1 & \cdots & t_i \\
   d_1 & \cdots & d_i
  \end{array}\!\!
  \right)}}
$ be non-zero elements of the semigroup $\mathscr{I}_\lambda^n(S)$. Then the following conditions hold:
\begin{itemize}
  \item[$(i)$] $\alpha_S\mathscr{R}\beta_S$ in $\mathscr{I}_\lambda^n(S)$ if and only if there exists a permutation $\sigma\colon\{1,\ldots,i\}\rightarrow\{1,\ldots,i\}$ such that $a_1=c_{(1)\sigma}$, $\ldots, a_i=c_{(i)\sigma}$ and $s_1\mathscr{R}t_{(1)\sigma}$, $\ldots, s_i\mathscr{R}t_{(i)\sigma}$ in $S$;
  \item[$(ii)$] $\alpha_S\mathscr{L}\beta_S$ in $\mathscr{I}_\lambda^n(S)$ if and only if there exists a permutation $\sigma\colon\{1,\ldots,i\}\rightarrow\{1,\ldots,i\}$ such that $b_1=d_{(1)\sigma}$, $\ldots, b_i=d_{(i)\sigma}$ and $s_1\mathscr{L}t_{(1)\sigma}$, $\ldots, s_i\mathscr{L}t_{(i)\sigma}$ in $S$;
  \item[$(iii)$] $\alpha_S\mathscr{D}\beta_S$ in $\mathscr{I}_\lambda^n(S)$ if and only if there exists a permutation $\sigma\colon\{1,\ldots,i\}\rightarrow\{1,\ldots,i\}$ such that $s_1\mathscr{D}t_{(1)\sigma}$, $\ldots, s_i\mathscr{D}t_{(i)\sigma}$ in $S$;
  \item[$(iv)$] $\alpha_S\mathscr{H}\beta_S$ in $\mathscr{I}_\lambda^n(S)$ if and only if there exists a permutations $\sigma,\rho\colon\{1,\ldots,i\}\rightarrow\{1,\ldots,i\}$ such that $s_1\mathscr{R}t_{(1)\sigma}$, $\ldots, s_i\mathscr{R}t_{(i)\sigma}$ and $s_1\mathscr{L}t_{(1)\rho}$, $\ldots, s_i\mathscr{L}t_{(i)\rho}$ in $S$;
  \item[$(v)$] $\alpha_S\mathscr{J}\beta_S$ in $\mathscr{I}_\lambda^n(S)$ if and only if there exists a permutation $\pi\colon\{1,\ldots,i\}\rightarrow\{1,\ldots,i\}$ such that $s_1\mathscr{J}t_{(1)\pi}$, $\ldots, s_i\mathscr{J}t_{(i)\pi}$ in $S$.
\end{itemize}
\end{proposition}

\begin{proof}
$(i)$ $(\Rightarrow)$ Suppose that $\alpha_S\mathscr{R}\beta_S$ in $\mathscr{I}_\lambda^n(S)$. Then there exist non-zero elements $\gamma_S={\small{\left(\!\!
  \begin{array}{ccc}
   e_1 & \cdots & e_k \\
   u_1 & \cdots & u_k \\
   f_1 & \cdots & f_k
  \end{array}\!\!
  \right)}}$
and $\delta_S={\small{\left(\!\!
  \begin{array}{ccc}
   g_1 & \cdots & g_j \\
   v_1 & \cdots & v_j \\
   h_1 & \cdots & h_j
  \end{array}\!\!
  \right)}}$
of the semigroup $\mathscr{I}_\lambda^n(S)$ such that $\alpha_S=\beta_S\gamma_S$, $\beta_S=\alpha_S\delta_S$, $i\leqslant j\leqslant n$ and $i\leqslant k\leqslant n$. Also, the definition of the semigroup operation of $\mathscr{I}_\lambda^n(S)$ implies that without loss of generality we may assume that $j=k=i$. Then the equalities $\alpha_S=\beta_S\gamma_S$ and $\beta_S=\alpha_S\delta_S$ imply that $\{a_1, \ldots, a_i\}=\{c_1,\ldots, c_i\}$, $\{b_1,\ldots, b_i\}=\{g_1,\ldots, g_i\}$ and $\{d_1, \ldots, d_i\}=\{e_1,\ldots, e_i\}$. Now, the semigroup operation of $\mathscr{I}_\lambda^n(S)$ implies that there exist permutations $\sigma,\rho,\zeta\colon\{1,\ldots,i\}\rightarrow\{1,\ldots,i\}$ such that $a_1=c_{(1)\sigma}$, $\ldots, a_i=c_{(i)\sigma}$, $d_1=e_{(1)\rho}$, $\ldots, d_i=e_{(i)\rho}$, and $b_1=g_{(1)\zeta}, \ldots, b_i=g_{(i)\zeta}$, and hence we have that
\begin{equation*}
\begin{split}
    {\small{\left(\!\!
  \begin{array}{ccc}
   a_1 & \cdots & a_i \\
   s_1 & \cdots & s_i \\
   b_1 & \cdots & b_i
  \end{array}\!\!
  \right)}}{=}&
  {\small{\left(\!\!
  \begin{array}{ccc}
   c_1 & \cdots & c_i \\
   t_1 & \cdots & t_i \\
   d_1 & \cdots & d_i
  \end{array}\!\!
  \right)}}
  {\cdot}
  {\small{\left(\!\!
  \begin{array}{ccc}
   e_1 & \cdots & e_i \\
   u_1 & \cdots & u_i \\
   f_1 & \cdots & f_i
  \end{array}\!\!
  \right)}}{=}
  {\small{\left(\!\!
  \begin{array}{ccc}
   c_1 & \cdots & c_i \\
   t_1 & \cdots & t_i \\
   d_1 & \cdots & d_i
  \end{array}\!\!
  \right)}}
  {\cdot}
  {\small{\left(\!\!
  \begin{array}{ccc}
   d_1         & \cdots & d_i \\
   u_{(1)\rho} & \cdots & u_{(i)\rho} \\
   f_{(1)\rho} & \cdots & f_{(i)\rho}
  \end{array}\!\!
  \right)}}{=}
  {\small{\left(\!\!
  \begin{array}{ccc}
   c_1            & \cdots & c_i \\
   t_1u_{(1)\rho} & \cdots & t_iu_{(i)\rho} \\
   f_{(1)\rho}    & \cdots & f_{(i)\rho}
  \end{array}\!\!
  \right)}}=\\
   = &
   {\small{\left(\!\!
  \begin{array}{ccc}
   a_1                              & \cdots & a_i \\
   t_{(1)\sigma}u_{((1)\rho)\sigma} & \cdots & t_{(i)\sigma}u_{((i)\rho)\sigma} \\
   f_{((1)\rho)\sigma}              & \cdots & f_{((i)\rho)\sigma}
  \end{array}\!\!
  \right)}}
\end{split}
\end{equation*}
and
\begin{equation*}
\begin{split}
  {\small{\left(\!\!
  \begin{array}{ccc}
   c_1 & \cdots & c_i \\
   t_1 & \cdots & t_i \\
   d_1 & \cdots & d_i
  \end{array}\!\!
  \right)}}
  {=}&
    {\small{\left(\!\!
  \begin{array}{ccc}
   a_1 & \cdots & a_i \\
   s_1 & \cdots & s_i \\
   b_1 & \cdots & b_i
  \end{array}\!\!
  \right)}}
  {\cdot}
  {\small{\left(\!\!
  \begin{array}{ccc}
   g_1 & \cdots & g_i \\
   v_1 & \cdots & v_i \\
   h_1 & \cdots & h_i
  \end{array}\!\!
  \right)}}{=}
  {\small{\left(\!\!
  \begin{array}{ccc}
   a_1 & \cdots & a_i \\
   s_1 & \cdots & s_i \\
   b_1 & \cdots & b_i
  \end{array}\!\!
  \right)}}
  {\cdot}
  {\small{\left(\!\!
  \begin{array}{ccc}
   b_1          & \cdots & b_i \\
   v_{(1)\zeta} & \cdots & v_{(i)\zeta} \\
   h_{(1)\zeta} & \cdots & h_{(i)\zeta}
  \end{array}\!\!
  \right)}}{=} 
  {\small{\left(\!\!
  \begin{array}{ccc}
   a_1             & \cdots & a_i \\
   s_1v_{(1)\zeta} & \cdots & s_iv_{(i)\zeta} \\
   h_{(1)\zeta}    & \cdots & h_{(i)\zeta}
  \end{array}\!\!
  \right)}}=\\
  = &
  {\small{\left(\!\!
  \begin{array}{ccc}
   c_1                                         & \cdots & c_i \\
   s_{(1)\sigma^{-1}}v_{((1)\zeta)\sigma^{-1}} & \cdots &
   s_{(i)\sigma^{-1}}v_{((i)\zeta)\sigma^{-1}} \\
   h_{((1)\zeta)\sigma^{-1}}                   & \cdots  & h_{((i)\zeta)\sigma^{-1}}
  \end{array}\!\!
  \right)}}.
\end{split}
\end{equation*}
Therefore we get that
\begin{equation}\label{eq-1}
   s_1=t_{(1)\sigma}u_{((1)\rho)\sigma}, \quad \ldots , \quad s_i=t_{(i)\sigma}u_{((i)\rho)\sigma}, \qquad \hbox{and} \qquad
   t_1=s_{(1)\sigma^{-1}}v_{((1)\zeta)\sigma^{-1}}, \; \ldots \; t_i=s_{(i)\sigma^{-1}}v_{((i)\zeta)\sigma^{-1}}.
\end{equation}
Since $\sigma\colon\{1,\ldots,i\}\rightarrow\{1,\ldots,i\}$ is a permutation conditions (\ref{eq-1}) imply that $s_1\mathscr{R}t_{(1)\sigma}$, $\ldots,$ $s_i\mathscr{R}t_{(i)\sigma}$ in $S$.

$(\Leftarrow)$ Suppose that for $\alpha_S,\beta_S\in\mathscr{I}_\lambda^n(S)$ there exists a permutation $\sigma\colon\{1,\ldots,i\}\rightarrow\{1,\ldots,i\}$ such that $a_1=c_{(1)\sigma}$, $\ldots, a_i=c_{(i)\sigma}$ and $s_1\mathscr{R}t_{(1)\sigma}$, $\ldots, s_i\mathscr{R}t_{(i)\sigma}$ in $S$. Then there exist $u_1,\ldots,u_i,v_1,\ldots,v_i\in S^1$ such that
\begin{equation*}
s_1=t_{(1)\sigma}u_{1}, \quad \ldots , \quad s_i=t_{(i)\sigma}u_{i}, \quad t_1=s_{(1)\sigma^{-1}}v_1, \quad \ldots \quad t_i=s_{(i)\sigma^{-1}}v_i.
\end{equation*}
Thus we get that
\begin{equation*}
  {\small{\left(\!\!
  \begin{array}{ccc}
   a_1 & \cdots & a_i \\
   s_1 & \cdots & s_i \\
   b_1 & \cdots & b_i
  \end{array}\!\!
  \right)}}{=}
  {\small{\left(\!\!
  \begin{array}{ccc}
   c_{(1)\sigma}      & \cdots & c_{(i)\sigma} \\
   t_{(1)\sigma}u_{1} & \cdots & t_{(i)\sigma}u_{i} \\
   b_1                & \cdots & b_i
  \end{array}\!\!
  \right)}}{=}
  {\small{\left(\!\!
  \begin{array}{ccc}
   c_1                   & \cdots & c_i \\
   t_1u_{(1)\sigma^{-1}} & \cdots & t_iu_{(i)\sigma^{-1}} \\
   b_{(1)\sigma^{-1}}    & \cdots & b_{(i)\sigma^{-1}}
  \end{array}\!\!
  \right)}}{=}
  {\small{\left(\!\!
  \begin{array}{ccc}
   c_1 & \cdots & c_i \\
   t_1 & \cdots & t_i \\
   d_1 & \cdots & d_i
  \end{array}\!\!
  \right)}}{\cdot}
  {\small{\left(\!\!
  \begin{array}{ccc}
   d_1                & \cdots & d_i \\
   u_{(1)\sigma^{-1}} & \cdots & u_{(i)\sigma^{-1}} \\
   b_{(1)\sigma^{-1}} & \cdots & b_{(i)\sigma^{-1}}
  \end{array}\!\!
  \right)}}
\end{equation*}
and
\begin{equation*}
  {\small{\left(\!\!
  \begin{array}{ccc}
   c_1 & \cdots & c_i \\
   t_1 & \cdots & t_i \\
   d_1 & \cdots & d_i
  \end{array}\!\!
  \right)}}{=}
  {\small{\left(\!\!
  \begin{array}{ccc}
   a_{(1)\sigma^{-1}}    & \cdots & a_{(i)\sigma^{-1}} \\
   s_{(1)\sigma^{-1}}v_1 & \cdots & s_{(i)\sigma^{-1}}v_i \\
   d_1                   & \cdots & d_i
  \end{array}\!\!
  \right)}}{=}
  {\small{\left(\!\!
  \begin{array}{ccc}
   a_1              & \cdots & a_i \\
   s_1v_{(1)\sigma} & \cdots & s_i v_{(i)\sigma} \\
   d_{(1)\sigma}    & \cdots & d_{(i)\sigma}
  \end{array}\!\!
  \right)}}{=}
  {\small{\left(\!\!
  \begin{array}{ccc}
   a_1 & \cdots & a_i \\
   s_1 & \cdots & s_i \\
   b_1 & \cdots & b_i
  \end{array}\!\!
  \right)}}
  {\cdot}
  {\small{\left(\!\!
  \begin{array}{ccc}
   b_1           & \cdots & b_i \\
   v_{(1)\sigma} & \cdots & v_{(i)\sigma} \\
   d_{(1)\sigma} & \cdots & d_{(i)\sigma}
  \end{array}\!\!
  \right)}},
\end{equation*}
and hence $\alpha_S\mathscr{R}\beta_S$ in $\mathscr{I}_\lambda^n(S)$.

The proof of statement $(ii)$ is similar to the proof of $(i)$.

$(iii)$ $(\Rightarrow)$ Suppose that $\alpha_S\mathscr{D}\beta_S$ in $\mathscr{I}_\lambda^n(S)$. Then there exist a non-zero element  $\gamma_S={\small{\left(\!\!
  \begin{array}{ccc}
   e_1 & \cdots & e_i \\
   u_1 & \cdots & u_i \\
   f_1 & \cdots & f_i
  \end{array}\!\!
  \right)}}$
of the semigroup $\mathscr{I}_\lambda^n(S)$ such that $\alpha_S\mathscr{R}\gamma_S$ and $\gamma_S\mathscr{L}\beta_S$ in $\mathscr{I}_\lambda^n(S)$. By statement $(i)$ there exists a permutation $\zeta\colon\{1,\ldots,i\}\rightarrow\{1,\ldots,i\}$ such that $e_1=a_{(1)\zeta}$, $\ldots, e_i=a_{(i)\zeta}$ and $u_1\mathscr{R}s_{(1)\zeta}$, $\ldots, u_i\mathscr{R}s_{(i)\zeta}$ in $S$ and by statement $(ii)$ there exists a permutation $\varsigma\colon\{1,\ldots,i\}\rightarrow\{1,\ldots,i\}$ such that $f_1=d_{(1)\varsigma},\ldots, f_i=d_{(i)\varsigma}$ and $u_1\mathscr{L}s_{(1)\varsigma}$, $\ldots, u_i\mathscr{L}s_{(i)\varsigma}$ in $S$. This implies that $s_1\mathscr{D}t_{(1)\sigma}$, $\ldots, s_i\mathscr{D}t_{(i)\sigma}$ in $S$ for the permutation $\sigma=\zeta\circ\varsigma^{-1}$ of $\{1,\ldots,i\}$.

$(\Leftarrow)$ Suppose that there exists a permutation $\sigma\colon\{1,\ldots,i\}\rightarrow\{1,\ldots,i\}$ such that $s_1\mathscr{D}t_{(1)\sigma}$, $\ldots,$ $s_i\mathscr{D}t_{(i)\sigma}$ in $S$. Then the definition of the relation $\mathscr{D}$ implies that there exist $u_1,\ldots,u_i\in S$ such that $s_1\mathscr{R}u_1$, $\ldots,$ $s_i\mathscr{R}u_i$ and $u_1\mathscr{L}t_{(1)\sigma}$, $\ldots,$ $u_i\mathscr{L}t_{(i)\sigma}$ in $S$. Now, for the element $\gamma_S={\small{\left(\!\!
  \begin{array}{ccc}
   a_1           & \cdots & a_i \\
   u_1           & \cdots & u_i \\
   d_{(1)\sigma} & \cdots & d_{(i)\sigma}
  \end{array}\!\!
  \right)}}$
of the semigroup $\mathscr{I}_\lambda^n(S)$ by statements $(i)$ and $(ii)$ we have that $\alpha_S\mathscr{R}\gamma_S$ and $\gamma_S\mathscr{L}\beta_S$ in $\mathscr{I}_\lambda^n(S)$.

$(iv)$ follows from statements $(i)$ and $(ii)$.

$(v)$ $(\Rightarrow)$ Suppose that $\alpha_S\mathscr{J}\beta_S$ in $\mathscr{I}_\lambda^n(S)$. Then there exist non-zero elements $\gamma_S^l={\small{\left(\!\!
  \begin{array}{ccc}
   e_1^l & \cdots & e_{k_l}^l \\
   u_1^l & \cdots & u_{k_l}^l \\
   f_1^l & \cdots & f_{k_l}^l
  \end{array}\!\!
  \right)}}$,
$\gamma_S^r={\small{\left(\!\!
  \begin{array}{ccc}
   e_1^r & \cdots & e_{k_r}^r \\
   u_1^r & \cdots & u_{k_r}^r \\
   f_1^r & \cdots & f_{k_r}^r
  \end{array}\!\!
  \right)}}$,
$\delta_S^l={\small{\left(\!\!
  \begin{array}{ccc}
   g_1^l & \cdots & g_{j_l}^l \\
   v_1^l & \cdots & v_{j_l}^l \\
   h_1^l & \cdots & h_{j_l}^l
  \end{array}\!\!
  \right)}}$
and $\delta_S^r={\small{\left(\!\!
  \begin{array}{ccc}
   g_1^r & \cdots & g_{j_r}^r \\
   v_1^r & \cdots & v_{j_r}^r \\
   h_1^r & \cdots & h_{j_r}^r
  \end{array}\!\!
  \right)}}$
of the semigroup $\mathscr{I}_\lambda^n(S)$ such that $\alpha_S=\gamma_S^l\beta_S\gamma_S^r$, $\beta_S=\delta_S^l\alpha_S\delta_S^r$ and $i\leqslant k_l,k_r,j_l,j_r \leqslant n$ (see \cite{GreenJ1951} or \cite[Section~II.1]{Grillet1995}). Also, the definition of the semigroup operation of $\mathscr{I}_\lambda^n(S)$ implies that without loss of generality we may assume that $k_l=k_r=j_l=j_r=i$. Then the equalities $\alpha_S=\gamma_S^l\beta_S\gamma_S^r$ and $\beta_S=\delta_S^l\alpha_S\delta_S^r$ imply that $\{a_1,\ldots, a_i\}=\{g_1^l, \ldots, g_i^l\}=\{h_1^l, \ldots, h_i^l\}$, $\{b_1,\ldots, b_i\}=\{f_1^r, \ldots, f_i^r\}=\{g_1^r, \ldots, g_i^r\}$, $\{c_1, \ldots, c_i\}=\{g_1^l, \ldots, g_i^l\}=\{f_1^l, \ldots, f_i^l\}$ and $\{d_1, \ldots, d_i\}=\{e_1^r, \ldots, e_i^r\}=\{h_1^r, \ldots, h_i^r\}$. Now, the semigroup operation of $\mathscr{I}_\lambda^n(S)$ implies that there exist permutations $\sigma,\rho,\zeta,\varsigma,\nu,\kappa \colon\{1,\ldots,i\}\rightarrow\{1,\ldots,i\}$ such that $a_1=e^l_{(1)\sigma}$, $\ldots, a_i=e^l_{(i)\sigma}$, $c_1=f^l_{(1)\rho}$, $\ldots, c_i=f^l_{(i)\rho}$, $d_1=e^r_{(1)\zeta}$, $\ldots, d_i=e^r_{(i)\zeta}$, $c_1=g^l_{(1)\varsigma}$, $\ldots, c_i=g^l_{(i)\varsigma}$, $a_1=h^l_{(1)\nu}$, $\ldots, a_i=h^l_{(i)\nu}$ and $b_1=g^r_{(1)\kappa}$, $\ldots, b_i=g^r_{(i)\kappa}$, and hence we have that
\begin{equation*}
\begin{split}
    &{\small{\left(\!\!
  \begin{array}{ccc}
   a_1 & \cdots & a_i \\
   s_1 & \cdots & s_i \\
   b_1 & \cdots & b_i
  \end{array}\!\!
  \right)}}
  {=}
  {\small{\left(\!\!
  \begin{array}{ccc}
   e_1^l & \cdots & e_{k_l}^l \\
   u_1^l & \cdots & u_{k_l}^l \\
   f_1^l & \cdots & f_{k_l}^l
  \end{array}\!\!
  \right)}}
  {\cdot}
  {\small{\left(\!\!
  \begin{array}{ccc}
   c_1 & \cdots & c_i \\
   t_1 & \cdots & t_i \\
   d_1 & \cdots & d_i
  \end{array}\!\!
  \right)}}
  {\cdot}
  {\small{\left(\!\!
  \begin{array}{ccc}
   e_1^r & \cdots & e_{k_r}^r \\
   u_1^r & \cdots & u_{k_r}^r \\
   f_1^r & \cdots & f_{k_r}^r
  \end{array}\!\!
  \right)}}{=}
  \\
  &=
  {\small{\left(\!\!
  \begin{array}{ccc}
   e_{(1)\rho}^l & \cdots & e_{(i)\rho}^l \\
   u_{(1)\rho}^l & \cdots & u_{(i)\rho}^l \\
   c_1           & \cdots & c_i
  \end{array}\!\!
  \right)}}
  {\cdot}
  {\small{\left(\!\!
  \begin{array}{ccc}
   c_1 & \cdots & c_i \\
   t_1 & \cdots & t_i \\
   d_1 & \cdots & d_i
  \end{array}\!\!
  \right)}}
  {\cdot}
  {\small{\left(\!\!
  \begin{array}{ccc}
   d_1            & \cdots & d_i \\
   u_{(1)\zeta}^r & \cdots & u_{(i)\zeta}^r \\
   f_{(1)\zeta}^r & \cdots & f_{(i)\zeta}^r
  \end{array}\!\!
  \right)}}=
  {\small{\left(\!\!
  \begin{array}{ccc}
   e_{(1)\rho}^l                  & \cdots & e_{(i)\rho}^l \\
   u_{(1)\rho}^lt_1u_{(1)\zeta}^r & \cdots & u_{(i)\rho}^lt_iu_{(i)\zeta}^r \\
   f_{(1)\zeta}^r                 & \cdots & f_{(i)\zeta}^r
  \end{array}\!\!
  \right)}}=\\
  &=
  {\small{\left(\!\!
  \begin{array}{ccc}
   e_1^l                                          & \cdots & e_i^l \\
   u_1^lt_{(1)\rho^{-1}}u_{((1)\zeta)\rho^{-1}}^r & \cdots & u_1^lt_{(i)\rho^{-1}}u_{((i)\zeta)\rho^{-1}}^r \\
   f_{((1)\zeta)\rho^{-1}}^r                      & \cdots &
   f_{((i)\zeta)\rho^{-1}}^r
  \end{array}\!\!
  \right)}}{=}
  {\small{\left(\!\!
  \begin{array}{ccc}
   a_1                               & \cdots & a_i \\
   u_{(1)\sigma}^lt_{((1)\rho^{-1})\sigma}u_{(((1)\zeta)\rho^{-1})\sigma}^r & \cdots & u_{(i)\sigma}^lt_{((i)\rho^{-1})\sigma}u_{(((i)\zeta)\rho^{-1})\sigma}^r \\
   f_{(((1)\zeta)\rho^{-1})\sigma}^r & \cdots &
   f_{(((i)\zeta)\rho^{-1})\sigma}^r
  \end{array}\!\!
  \right)}}
\end{split}
\end{equation*}
and
\begin{equation*}
\begin{split}
    &{\small{\left(\!\!
  \begin{array}{ccc}
   c_1 & \cdots & c_i \\
   t_1 & \cdots & t_i \\
   d_1 & \cdots & d_i
  \end{array}\!\!
  \right)}}=
  {\small{\left(\!\!
  \begin{array}{ccc}
   g_1^l & \cdots & g_i^l \\
   v_1^l & \cdots & v_i^l \\
   h_1^l & \cdots & h_i^l
  \end{array}\!\!
  \right)}}
  \cdot
  {\small{\left(\!\!
  \begin{array}{ccc}
   a_1 & \cdots & a_i \\
   s_1 & \cdots & s_i \\
   b_1 & \cdots & b_i
  \end{array}\!\!
  \right)}}
  \cdot
  {\small{\left(\!\!
  \begin{array}{ccc}
   g_1^r & \cdots & g_i^r \\
   v_1^r & \cdots & v_i^r \\
   h_1^r & \cdots & h_i^r
  \end{array}\!\!
  \right)}}=
  \\
  &=
  {\small{\left(\!\!
  \begin{array}{ccc}
   g_1^l & \cdots & g_i^l \\
   v_1^l & \cdots & v_i^l \\
   h_1^l & \cdots & h_i^l
  \end{array}\!\!
  \right)}}
  \cdot
  {\small{\left(\!\!
  \begin{array}{ccc}
   a_1 & \cdots & a_i \\
   s_1 & \cdots & s_i \\
   b_1 & \cdots & b_i
  \end{array}\!\!
  \right)}}
  \cdot
  {\small{\left(\!\!
  \begin{array}{ccc}
   b_1             & \cdots & b_i \\
   v_{(1)\kappa}^r & \cdots & v_{(i)\kappa}^r \\
   h_{(1)\kappa}^r & \cdots & h_{(i)\kappa}^r
  \end{array}\!\!
  \right)}}=
  {\small{\left(\!\!
  \begin{array}{ccc}
   g_{(1)\nu}^l                   & \cdots & g_{(i)\nu}^l \\
   v_{(1)\nu}^ls_1v_{(1)\kappa}^r & \cdots & v_{(i)\nu}^ls_iv_{(i)\kappa}^r \\
   h_{(1)\kappa}^r                & \cdots & h_{(i)\kappa}^r
  \end{array}\!\!
  \right)}}{=}
  \\
  &=
  {\small{\left(\!\!\!\!
  \begin{array}{ccc}
   g_1^l                                         & \cdots & g_i^l \\
   v_1^ls_{(1)\nu^{-1}}v_{((1)\kappa)\nu^{-1}}^r & \cdots & v_i^ls_{(i)\nu^{-1}}v_{((i)\kappa)\nu^{-1}}^r \\
   h_{((1)\kappa)\nu^{-1}}^r                     & \cdots & h_{((i)\kappa)\nu^{-1}}^r
  \end{array}\!\!\!\!
  \right)}}{=}
  {\small{\left(\!\!
  \begin{array}{ccc}
   c_1 & \cdots & c_i \\
   v_{(1)\zeta}^ls_{((1)\nu^{-1})\varsigma}v_{(((1)\kappa)\nu^{-1})\varsigma}^r & \cdots & v_{(i)\zeta}^ls_{((1)\nu^{-1})\varsigma}v_{(((i)\kappa)\nu^{-1})\varsigma}^r \\
   h_{(((1)\kappa)\nu^{-1})\varsigma}^r & \cdots &
   h_{(((i)\kappa)\nu^{-1})\varsigma}^r
  \end{array}\!\!
  \right)}}
\end{split}
\end{equation*}
Then the definition of the semigroup $\mathscr{I}_\lambda^n(S)$ implies the following equalities
\begin{equation*}
   d_1=h_{(((1)\kappa)\nu^{-1})\varsigma}^r, \qquad \ldots \; , \qquad
   d_i=h_{(((i)\kappa)\nu^{-1})\varsigma}^r.
\end{equation*}
Now, by the equality $\alpha_S=\gamma_S^l\beta_S\gamma_S^r$ we get that
\begin{equation*}
\begin{split}
  &{\small{\left(\!\!
  \begin{array}{ccc}
   a_1 & \cdots & a_i \\
   s_1 & \cdots & s_i \\
   b_1 & \cdots & b_i
  \end{array}\!\!
  \right)}}
{=}
  {\small{\left(\!\!
  \begin{array}{ccc}
   e_1^l & \cdots & e_{k_l}^l \\
   u_1^l & \cdots & u_{k_l}^l \\
   f_1^l & \cdots & f_{k_l}^l
  \end{array}\!\!
  \right)}}
  {\cdot}
  {\small{\left(\!\!
  \begin{array}{ccc}
   c_1 & \cdots & c_i \\
   t_1 & \cdots & t_i \\
   d_1 & \cdots & d_i
  \end{array}\!\!
  \right)}}
  {\cdot}
  {\small{\left(\!\!
  \begin{array}{ccc}
   e_1^r & \cdots & e_{k_r}^r \\
   u_1^r & \cdots & u_{k_r}^r \\
   f_1^r & \cdots & f_{k_r}^r
  \end{array}\!\!
  \right)}}{=}\\
  &=
  {\small{\left(\!\!
  \begin{array}{ccc}
   e_1^l & \cdots & e_{k_l}^l \\
   u_1^l & \cdots & u_{k_l}^l \\
   f_1^l & \cdots & f_{k_l}^l
  \end{array}\!\!
  \right)}}{\cdot}
  {\small{\left(\!\!\!
  \begin{array}{ccc}
   c_1  & \cdots & c_i \\
   v_{(1)\zeta}^ls_{((1)\nu^{-1})\varsigma}v_{(((1)\kappa)\nu^{-1})\varsigma}^r & \cdots & v_{(i)\zeta}^ls_{((1)\nu^{-1})\varsigma}v_{(((i)\kappa)\nu^{-1})\varsigma}^r \\
   d_1 & \cdots & d_i
  \end{array}\!\!\!
  \right)}}
  {\cdot}
  {\small{\left(\!\!
  \begin{array}{ccc}
   e_1^r & \cdots & e_{k_r}^r \\
   u_1^r & \cdots & u_{k_r}^r \\
   f_1^r & \cdots & f_{k_r}^r
  \end{array}\!\!
  \right)}}=
  \\
  &=
  {\small{\left(\!\!
  \begin{array}{ccc}
   e_{(1)\rho}^l & \cdots & e_{(i)\rho}^l \\
   u_{(1)\rho}^l & \cdots & u_{(i)\rho}^l \\
   c_1           & \cdots & c_i
  \end{array}\!\!
  \right)}}{\cdot}
  {\small{\left(\!\!\!\!
  \begin{array}{ccc}
   c_1 & {\cdots} & c_i \\
   v_{(1)\zeta}^ls_{((1)\nu^{-1})\varsigma}v_{(((1)\kappa)\nu^{-1})\varsigma}^r & {\cdots} & v_{(i)\zeta}^ls_{((1)\nu^{-1})\varsigma}v_{(((i)\kappa)\nu^{-1})\varsigma}^r \\
   d_1 & {\cdots} & d_i
  \end{array}\!\!\!\!
  \right)}}
  {\cdot}
  {\small{\left(\!\!
  \begin{array}{ccc}
   d_1            & \cdots & d_i \\
   u_{(1)\zeta}^r & \cdots & u_{(i)\zeta}^r \\
   f_{(1)\zeta}^r & \cdots & f_{(i)\zeta}^r
  \end{array}\!\!
  \right)}}{=}
  \\
  &=
  {\small{\left(\!\!\!\!
  \begin{array}{ccc}
   e_{(1)\rho}^l & {\cdots} & e_{(i)\rho}^l \\
   u_{(1)\rho}^lv_{(1)\zeta}^l s_{((1)\nu^{-1})\varsigma}v_{(((1)\kappa)\nu^{-1})\varsigma}^r u_{(1)\zeta}^r & {\cdots} & u_{(i)\rho}^lv_{(i)\zeta}^l s_{((1)\nu^{-1})\varsigma}v_{(((i)\kappa)\nu^{-1})\varsigma}^r u_{(1)\zeta}^r\\
   f_{(1)\zeta}^r & {\cdots} & f_{(i)\zeta}^r
  \end{array}\!\!\!\!
  \right)}}
\end{split}
\end{equation*}
which implies the following equalities
\begin{align*}
    &s_1=u_{(1)\sigma}^lv_{(((1)\zeta)\rho^{-1})\sigma}^l s_{((((1)\nu^{-1})\varsigma)\rho^{-1})\sigma}
   v_{(((((1)\kappa)\nu^{-1})\varsigma)\rho^{-1})\sigma}^r u_{(((1)\zeta)\rho^{-1})\sigma}^r\\
   &\qquad\ldots\qquad\ldots\qquad\ldots\qquad\ldots\qquad \ldots\qquad\ldots\qquad\ldots\qquad\ldots\qquad\\
   &s_i=u_{(i)\sigma}^lv_{(((i)\zeta)\rho^{-1})\sigma}^l s_{((((i)\nu^{-1})\varsigma)\rho^{-1})\sigma}
   v_{(((((i)\kappa)\nu^{-1})\varsigma)\rho^{-1})\sigma}^r u_{(((i)\zeta)\rho^{-1})\sigma}^r.
\end{align*}
Hence for the permutation $\pi=\nu^{-1}\varsigma\rho^{-1}\sigma\colon\{1,\ldots,i\}\rightarrow\{1,\ldots,i\}$ we have that $s_1\mathscr{J}t_{(1)\pi}$, $\ldots, s_i\mathscr{J}t_{(i)\pi}$ in $S$.

$(\Leftarrow)$ Suppose that for elements $\alpha_S,\beta_S\in\mathscr{I}_\lambda^n(S)$ there exists a permutation $\sigma\colon\{1,\ldots,i\}\rightarrow\{1,\ldots,i\}$ such that $s_1\mathscr{J}t_{(1)\sigma}$, $\ldots, s_i\mathscr{J}t_{(i)\sigma}$ in $S$. Then there exist $u_1,\ldots,u_i,v_1,\ldots,v_i,x_1,\ldots,x_i,y_1,\ldots,y_i\in S^1$ such that
\begin{equation*}
s_1{=}x_1t_{(1)\sigma}u_{1}, \qquad \ldots \; , \qquad s_i{=}x_it_{(i)\sigma}u_{i}, \qquad t_1{=}y_1s_{(1)\sigma^{-1}}v_1, \qquad \ldots \; , \qquad t_i{=}y_is_{(i)\sigma^{-1}}v_i.
\end{equation*}
Thus, we have that
\begin{equation*}
\begin{split}
  {\small{\left(\!\!
  \begin{array}{ccc}
   a_1 & \cdots & a_i \\
   s_1 & \cdots & s_i \\
   b_1 & \cdots & b_i
  \end{array}\!\!
  \right)}}& =
  {\small{\left(\!\!
  \begin{array}{ccc}
   c_{(1)\sigma}         & \cdots & c_{(i)\sigma} \\
   x_1t_{(1)\sigma}u_{1} & \cdots & x_it_{(i)\sigma}u_{i} \\
   b_{(1)\sigma}         & \cdots & b_{(i)\sigma}
  \end{array}\!\!
  \right)}}=
  {\small{\left(\!\!
  \begin{array}{ccc}
   c_1                                     & \cdots & c_i \\
   x_{(1)\sigma^{-1}}t_1u_{(1)\sigma^{-1}} & \cdots & x_{(i)\sigma^{-1}}t_iu_{(i)\sigma^{-1}} \\
   b_1                                     & \cdots & b_i
  \end{array}\!\!
  \right)}}=  \\
  = &
  {\small{\left(\!\!
  \begin{array}{ccc}
   c_1                & \cdots & c_i \\
   x_{(1)\sigma^{-1}} & \cdots & x_{(i)\sigma^{-1}} \\
   c_1                & \cdots & c_i
  \end{array}\!\!
  \right)}}
  \cdot
  {\small{\left(\!\!
  \begin{array}{ccc}
   c_1 & \cdots & c_i \\
   t_1 & \cdots & t_i \\
   b_1 & \cdots & b_i
  \end{array}\!\!
  \right)}}
  \cdot
  {\small{\left(\!\!
  \begin{array}{ccc}
   b_1                & \cdots & b_i \\
   u_{(1)\sigma^{-1}} & \cdots & u_{(i)\sigma^{-1}} \\
   b_1                & \cdots & b_i
  \end{array}\!\!
  \right)}}
\end{split}
\end{equation*}
and
\begin{equation*}
\begin{split}
   {\small{\left(\!\!
  \begin{array}{ccc}
   c_1 & \cdots & c_i \\
   t_1 & \cdots & t_i \\
   d_1 & \cdots & d_i
  \end{array}\!\!
  \right)}}&=
  {\small{\left(\!\!
  \begin{array}{ccc}
   a_{(1)\sigma^{-1}}       & \cdots & a_{(i)\sigma^{-1}} \\
   y_1s_{(1)\sigma^{-1}}v_1 & \cdots & y_is_{(i)\sigma^{-1}}v_i \\
   d_{(1)\sigma^{-1}}       & \cdots & d_{(i)\sigma^{-1}}
  \end{array}\!\!
  \right)}}=
  {\small{\left(\!\!
  \begin{array}{ccc}
   a_1                           & \cdots & a_i \\
   y_{(1)\sigma}s_1v_{(1)\sigma} & \cdots & y_{(i)\sigma}s_i v_{(i)\sigma} \\
   d_1                           & \cdots & d_i
  \end{array}\!\!
  \right)}}=\\
  = &
  {\small{\left(\!\!
  \begin{array}{ccc}
   a_1           & \cdots & a_i \\
   y_{(1)\sigma} & \cdots & y_{(i)\sigma} \\
   a_1           & \cdots & a_i
  \end{array}\!\!
  \right)}}\cdot
  {\small{\left(\!\!
  \begin{array}{ccc}
   a_1 & \cdots & a_i \\
   s_1 & \cdots & s_i \\
   d_1 & \cdots & d_i
  \end{array}\!\!
  \right)}}\cdot
  {\small{\left(\!\!
  \begin{array}{ccc}
   d_1           & \cdots & d_i \\
   v_{(1)\sigma} & \cdots & v_{(i)\sigma} \\
   d_1           & \cdots & d_i
  \end{array}\!\!
  \right)}},
\end{split}
\end{equation*}
and hence we get that $\alpha_S\mathscr{J}\beta_S$ in $\mathscr{I}_\lambda^n(S)$.
\end{proof}

\begin{remark}\label{remark-2.9}
Proposition~\ref{proposition-2.8}$(iv)$ implies that if if there exists a permutation $\sigma\colon\{1,\ldots,i\}\rightarrow\{1,\ldots,i\}$ such that $s_1\mathscr{H}t_{(1)\sigma}$, $\ldots, s_i\mathscr{H}t_{(i)\sigma}$ in $S$ then $\alpha_S\mathscr{H}\beta_S$ in $\mathscr{I}_\lambda^n(S)$. But Example~\ref{example-2.10} implies that converse statement is not true.
\end{remark}

\begin{example}\label{example-2.10}
Let $\lambda$ be any cardinal $\geqslant 2$ and ${\mathscr{C}}(p,q)$ be the bicyclic monoid. The bicyclic monoid ${\mathscr{C}}(p,q)$ is the semigroup with
the identity $1$ generated by two elements $p$ and $q$ subjected
only to the condition $pq=1$. The distinct elements of
${\mathscr{C}}(p,q)$ are exhibited in the following useful array
\begin{equation*}
\begin{array}{ccccc}
  1      & p      & p^2    & p^3    & \cdots \\
  q      & qp     & qp^2   & qp^3   & \cdots \\
  q^2    & q^2p   & q^2p^2 & q^2p^3 & \cdots \\
  q^3    & q^3p   & q^3p^2 & q^3p^3 & \cdots \\
  \vdots & \vdots & \vdots & \vdots & \ddots \\
\end{array}
\end{equation*}
and the semigroup operation on ${\mathscr{C}}(p,q)$ is determined as
follows:
\begin{equation*}
    q^kp^l\cdot q^mp^n=q^{k+m-\min\{l,m\}}p^{l+n-\min\{l,m\}}.
\end{equation*}
We fix arbitrary distinct elements $a_1$ and $a_1$ of $\lambda$ and put
\begin{equation*}
    \alpha=
    \left(
      \begin{array}{ccc}
        a_1 & a_1 \\
        qp  & q^2p^2 \\
        a_1 & a_1 \\
      \end{array}
    \right)
\qquad \hbox{~and~} \qquad
\beta=
    \left(
      \begin{array}{ccc}
        a_1 & a_2 \\
        qp^2  & q^2p \\
        a_2 & a_1 \\
      \end{array}
    \right).
\end{equation*}
Then we have that
\begin{equation*}
    \alpha=
    \left(
      \begin{array}{ccc}
        a_1 & a_2 \\
        qp^2  & q^2p \\
        a_2 & a_1 \\
      \end{array}
    \right)
    \cdot
    \left(
      \begin{array}{ccc}
        a_1 & a_2 \\
        p  & q \\
        a_2 & a_1 \\
      \end{array}
    \right)
\qquad \hbox{~and~} \qquad
\beta=
\left(
      \begin{array}{ccc}
        a_1 & a_1 \\
        qp  & q^2p^2 \\
        a_1 & a_1 \\
      \end{array}
    \right)
    \cdot
    \left(
      \begin{array}{ccc}
        a_1 & a_2 \\
        p  & q \\
        a_2 & a_1 \\
      \end{array}
    \right)
\end{equation*}
and hence $\alpha\mathscr{R}\beta$ in $\mathscr{I}_\lambda^n(S)$, and similar we have that
\begin{equation*}
    \alpha=
    \left(
      \begin{array}{ccc}
        a_1 & a_2 \\
        p  & q \\
        a_2 & a_1 \\
      \end{array}
    \right)
    \cdot
    \left(
      \begin{array}{ccc}
        a_1 & a_2 \\
        qp^2  & q^2p \\
        a_2 & a_1 \\
      \end{array}
    \right)
\qquad \hbox{~and~} \qquad
\beta=
\left(
      \begin{array}{ccc}
        a_1 & a_2 \\
        p  & q \\
        a_2 & a_1 \\
      \end{array}
    \right)
    \cdot
\left(
      \begin{array}{ccc}
        a_1 & a_1 \\
        qp  & q^2p^2 \\
        a_1 & a_1 \\
      \end{array}
    \right)
\end{equation*}
and hence $\alpha\mathscr{L}\beta$ in $\mathscr{I}_\lambda^n(S)$. Thus $\alpha\mathscr{H}\beta$ in $\mathscr{I}_\lambda^n(S)$, but the elements $qp$ and $q^2p^2$ are not pairwise $\mathscr{H}$-equivalent to $qp^2$ and $q^2p$ for any permutation $\sigma\colon\{1,2\}\rightarrow\{1,2\}$.
\end{example}

Recall \cite{KochWallace1957}, a semigroup $S$ is said to be:
\begin{enumerate}
\item[(a)] \emph{left stable} if for $a, b\in S$, $Sa\subseteq Sab$ implies $Sa=Sab$;
\item[(b)] \emph{right stable} if for $c, d\in S$, $cS\subseteq dcS$ implies $cS=dcS$;
\item[(b)] \emph{stable} if it is both left and right stable.
\end{enumerate}

We observe that in the book \cite{CP} there given other definition of a stable semigroup, and these two notion are distinct. A semigroup stable in the sense of Koch and Wallace is always stable in the sense of the book \cite{CP}, but not conversely (see: \cite{O'Carroll}). For semigroups with an identity element and for regular semigroups these two definitions of stability coincide.

The following proposition states that the construction of the semigroup $\mathscr{I}_\lambda^n(S)$ preserves left an right stabilities.

\begin{proposition}\label{proposition-2.11}
For every semigroup $S$, any non-zero cardinal $\lambda$ and any positive integer $n\leqslant\lambda$ the following statements hold:
\begin{itemize}
  \item[$(i)$] $\mathscr{I}_\lambda^n(S)$ is right stable if and only if so is $S$;
  \item[$(ii)$] $\mathscr{I}_\lambda^n(S)$ is left stable if and only if so is $S$;
  \item[$(iii)$] $\mathscr{I}_\lambda^n(S)$ is stable if and only if so is $S$.
\end{itemize}
\end{proposition}

\begin{proof}
$(i)$ $(\Leftarrow)$ Suppose the semigroup $S$ is right stable and assume that $\alpha_S=
  {\small{\left(%
  \begin{array}{ccc}
   a_1 & \cdots & a_i \\
   s_1 & \cdots & s_i \\
   b_1 & \cdots & b_i
  \end{array}%
  \right)}}
  $
and $\beta_S=
  {\small{\left(%
  \begin{array}{ccc}
   c_1 & \cdots & c_k \\
   t_1 & \cdots & t_k \\
   d_1 & \cdots & d_k
  \end{array}%
  \right)}}
$ are elements of the semigroup $\mathscr{I}_\lambda^n(S)$ such that $\alpha_S\mathscr{I}_\lambda^n(S)\subseteq \beta_S\alpha_S\mathscr{I}_\lambda^n(S)$. Then the above inclusion and the definition of the semigroup operation on $\mathscr{I}_\lambda^n(S)$ imply that $i\leqslant k$ and the following inclusion holds
\begin{equation*}
    \{a_1, \ldots, a_i\}\subseteq\{c_1, \ldots, c_k\}\cap \{d_1, \ldots, d_k\}.
\end{equation*}
Without loss of generality we may assume that $d_1=a_1$, $\ldots, d_i=a_i$. Then the inclusion $\alpha_S\mathscr{I}_\lambda^n(S)\subseteq \beta_S\alpha_S\mathscr{I}_\lambda^n(S)$ implies that there exists a permutation $\sigma\colon\{1,\ldots,i\}\rightarrow\{1,\ldots,i\}$ such that $c_1=a_{(1)\sigma}$, $\ldots, c_i=a_{(i)\sigma}$. Hence by the definition of the semigroup operation of $\mathscr{I}_\lambda^n(S)$ we get that
\begin{equation*}
\begin{split}
  \beta_S&\alpha_S\mathscr{I}_\lambda^n(S)
  {=}
  {\small{\left(\!\!
  \begin{array}{ccc}
   c_1 & \cdots & c_k \\
   t_1 & \cdots & t_k \\
   d_1 & \cdots & d_k
  \end{array}\!\!
  \right)}}
  {\cdot}
  {\small{\left(\!\!
  \begin{array}{ccc}
   a_1 & \cdots & a_i \\
   s_1 & \cdots & s_i \\
   b_1 & \cdots & b_i
  \end{array}\!\!
  \right)}}{\cdot}\mathscr{I}_\lambda^n(S){=}
  {\small{\left(\!\!%
  \begin{array}{cccccc}
   c_1 & \cdots & c_i & c_{i+1} & \cdots & c_k\\
   t_1 & \cdots & t_i & t_{i+1} & \cdots & t_k\\
   d_1 & \cdots & d_i & d_{i+1} & \cdots & d_k
  \end{array}\!\!%
  \right)}}
  {\cdot}
  {\small{\left(\!\!%
  \begin{array}{ccc}
   a_1 & \cdots & a_i \\
   s_1 & \cdots & s_i \\
   b_1 & \cdots & b_i
  \end{array}\!\!%
  \right)}}
  {\cdot}\mathscr{I}_\lambda^n(S){=}\\
  = &
  {\small{\left(\!\!%
  \begin{array}{cccccc}
   a_{(1)\sigma} & \cdots & a_{(i)\sigma} & c_{i+1} & \cdots & c_k\\
   t_1           & \cdots & t_i           & t_{i+1} & \cdots & t_k\\
   a_1           & \cdots & a_i           & d_{i+1} & \cdots & d_k
  \end{array}\!\!%
  \right)}}
  {\cdot}
  {\small{\left(\!\!%
  \begin{array}{ccc}
   a_1 & \cdots & a_i \\
   s_1 & \cdots & s_i \\
   b_1 & \cdots & b_i
  \end{array}\!\!%
  \right)}}
  {\cdot}\mathscr{I}_\lambda^n(S){=}
  {\small{\left(\!\!%
  \begin{array}{ccc}
   a_{(1)\sigma} & \cdots & a_{(i)\sigma} \\
   t_1           & \cdots & t_i           \\
   a_1           & \cdots & a_i
  \end{array}\!\!%
  \right)}}
  {\cdot}
  {\small{\left(\!\!%
  \begin{array}{ccc}
   a_1 & \cdots & a_i \\
   s_1 & \cdots & s_i \\
   b_1 & \cdots & b_i
  \end{array}\!\!%
  \right)}}
  {\cdot}\mathscr{I}_\lambda^n(S)
  {=}
  \\
%
  = &
  {\small{\left(\!\!%
  \begin{array}{cccc}
   a_{(1)\sigma} & \cdots & a_{(i)\sigma} \\
   t_1s_1        & \cdots & t_is_i        \\
   b_1           & \cdots & b_i
  \end{array}\!\!%
  \right)}}\cdot \mathscr{I}_\lambda^n(S){=}
  {\small{\left(\!\!%
  \begin{array}{ccc}
   a_1                                 & \cdots & a_i \\
   t_{(1)\sigma^{-1}}s_{(1)\sigma^{-1}}& \cdots & t_{(i)\sigma^{-1}}s_{(i)\sigma^{-1}}\\
   b_{(1)\sigma^{-1}}                  & \cdots & b_{(i)\sigma^{-1}}
  \end{array}\!\!%
  \right)}}\cdot \mathscr{I}_\lambda^n(S)=\\
  = & \{0\}\cup\bigcup \left\{ [t_{(1)\sigma^{-1}}s_{(1)\sigma^{-1}}S, \ldots,t_{(i)\sigma^{-1}}s_{(i)\sigma^{-1}}S]^{(a_1,\ldots,a_i)} _{(p_1,\ldots,p_i)} \colon p_1,\ldots,p_i\in\lambda \right\}\cup\\
    & \quad \cup\bigcup\Big\{ [t_{(l_1)\sigma^{-1}}s_{(l_1)\sigma^{-1}}S, \ldots,t_{(l_{i-1})\sigma^{-1}}s_{(l_{i-1})\sigma^{-1}}S]^{(l_1,\ldots,l_{i-1})} _{(p_1,\ldots,p_{i-1})} \colon l_1,\ldots,l_{i-1}\hbox{~are distinct elements}\\
    & \qquad \qquad \hbox{of~}\{1,\ldots,i\} \hbox{~and~} p_1,\ldots,p_{i-1}\in\lambda\Big\}\cup\cdots\cup\\
    & \quad \cup\bigcup\Big\{ [t_{(l)\sigma^{-1}}s_{(l)\sigma^{-1}}S]^{(l)} _{(p)} \colon l\in\{1,\ldots,i\} \hbox{~and~} p\in\lambda\Big\}
\end{split}
\end{equation*}
and
\begin{equation*}
\begin{split}
  \alpha_S&\mathscr{I}_\lambda^n(S)=
  {\small{\left(\!\!%
  \begin{array}{cccc}
   a_1 & \cdots & a_i \\
   s_1 & \cdots & s_i \\
   b_1 & \cdots & b_i
  \end{array}\!\!%
  \right)}}\cdot\mathscr{I}_\lambda^n(S)=
  \{0\}\cup\bigcup \left\{ [s_1S,\ldots,s_iS]^{(a_1,\ldots,a_i)} _{(p_1,\ldots,p_i)} \colon p_1,\ldots,p_i\in\lambda \right\}\cup\\
    & \cup\bigcup\Big\{ [s_{l_1}S, \ldots,s_{l_{i-1}}S]^{(l_1,\ldots,l_{i-1})} _{(p_1,\ldots,p_{i-1})} \colon l_1,\ldots,l_{i-1}\hbox{~are distinct elements of~}\{1,\ldots,i\}
    \hbox{~and~} p_1,\ldots,p_{i-1}\in\lambda\Big\}\cup\cdots\cup\\
    & \cup\bigcup\Big\{ [s_{l}S]^{(l)} _{(p)} \colon l\in\{1,\ldots,i\} \hbox{~and~} p\in\lambda\Big\}.
\end{split}
\end{equation*}
Hence, the inclusion $\alpha_S\mathscr{I}_\lambda^n(S)\subseteq \beta_S\alpha_S\mathscr{I}_\lambda^n(S)$ and semigroup operations of $\mathscr{I}_\lambda^n(S)$ and $S$ imply that $s_{l}S\subseteq t_{(l)\sigma^{-1}}s_{(l)\sigma^{-1}}S$, for every $l\in\{1,\ldots,i\}$. Since the semigroup of all permutations of a finite set is finite, we conclude that there exists a positive integer $j$ such that $\sigma^j\colon\{1,\ldots,i\}\rightarrow \{1,\ldots,i\}$ is an identity map and therefore we get that $\sigma^{j-1}=\sigma$. This implies that for every $l\in\{1,\ldots,i\}$ we have that
\begin{equation*}
\begin{split}
  s_{l}S\subseteq t_{(l)\sigma^{-1}}s_{(l)\sigma^{-1}}S\subseteq & \; t_{(l)\sigma^{-1}}t_{(l)\sigma^{-2}}s_{(l)\sigma^{-2}}S\subseteq\cdots\subseteq t_{(l)\sigma^{-1}}t_{(l)\sigma^{-2}}\cdots t_{(l)\sigma^{-j+1}}s_{(l)\sigma^{-j+1}}S= \\
    =&\; t_{(l)\sigma^{-1}}t_{(l)\sigma^{-2}}\cdots t_{l}s_{l}S.
\end{split}
\end{equation*}
Then the right stability of the semigroup $S$ implies the equality $s_{l}S=t_{(l)\sigma^{-1}}t_{(l)\sigma^{-2}}\cdots t_{l}s_{l}S$ and hence we have that $s_{l}S=t_{(l)\sigma^{-1}}s_{(l)\sigma^{-1}}S$, for every $l\in\{1,\ldots,i\}$. Then the inclusion $\alpha_S\mathscr{I}_\lambda^n(S)\subseteq \beta_S\alpha_S\mathscr{I}_\lambda^n(S)$ and above formulae imply the following equality $\alpha_S\mathscr{I}_\lambda^n(S)= \beta_S\alpha_S\mathscr{I}_\lambda^n(S)$, and hence the semigroup $\mathscr{I}_\lambda^n(S)$ is right stable.

$(\Rightarrow)$ Suppose that the semigroup $\mathscr{I}_\lambda^n(S)$ is right stable and $sS\subseteq tsS$ for $s, t\in S$. We fix an arbitrary $a\in \lambda$ and put $\alpha_S=
{\small{\left(%
  \begin{array}{c}
   a\\
   s\\
   a
  \end{array}%
  \right)}}$
and $\beta_S=
{\small{\left(%
  \begin{array}{c}
   a\\
   t\\
   a
  \end{array}%
  \right)}}$.
Hence by the definition of the semigroup operation of $\mathscr{I}_\lambda^n(S)$ we get that
\begin{equation*}
    \alpha_S\mathscr{I}_\lambda^n(S)=
    {\small{\left(%
  \begin{array}{c}
   a\\
   s\\
   a
  \end{array}%
  \right)}}
  \mathscr{I}_\lambda^n(S)=\{0\}\cup \bigcup\left\{ [sS]^{(a)} _{(p)} \colon p\in\lambda\right\}
\end{equation*}
and
\begin{equation*}
    \beta_S\alpha_S\mathscr{I}_\lambda^n(S)=
    {\small{\left(%
  \begin{array}{c}
   a\\
   t\\
   a
  \end{array}%
  \right)}}
    {\small{\left(%
  \begin{array}{c}
   a\\
   s\\
   a
  \end{array}%
  \right)}}
  \mathscr{I}_\lambda^n(S)=
  {\small{\left(%
  \begin{array}{c}
   a\\
   ts\\
   a
  \end{array}%
  \right)}}
  \mathscr{I}_\lambda^n(S)=\{0\}\cup \bigcup\left\{ [tsS]^{(a)} _{(p)} \colon p\in\lambda\right\},
\end{equation*}
and hence by the inclusion $sS\subseteq tsS$ we have that $\alpha_S\mathscr{I}_\lambda^n(S)\subseteq \beta_S\alpha_S\mathscr{I}_\lambda^n(S)$. Now the right stability of $\mathscr{I}_\lambda^n(S)$ implies the following equality $\alpha_S\mathscr{I}_\lambda^n(S)=\beta_S\alpha_S\mathscr{I}_\lambda^n(S)$. This implies $[sS]^{(a)} _{(p)}=[tsS]^{(a)} _{(p)}$ in $\mathscr{I}_\lambda^n(S)$ for every $p\in\lambda$, and hence $sS=tsS$.

The proof of statement $(ii)$ is dual to statement $(i)$.

$(iii)$ follows from statements $(i)$ and $(ii)$.
\end{proof}

\section{On semigroups with a tight ideal series}

Fix an arbitrary positive integer $m$ and any $p\in\{0,\ldots,m\}$. Let $A$ be a non-empty set and let $B$ be a non-empty proper subset of $A$. By $\left[B\subset A\right]^m_{p}$ we denote all elements $(x_1,\ldots,x_m)$ of the power $A^m$ which satisfy the following property: \emph{at most $p$ coordinates of $(x_1,\ldots,x_m)$ belong to $A\setminus B$.} It is obvious that $\left[B\subset A\right]^m_{m}=A^m$ any positive integer $m$, any non-empty set $A$ and any non-empty proper subset $B$ of $A$.

\smallskip

The above definition implies the following two lemmas.

\begin{lemma}\label{lemma-4.1}
Let $m$ be an arbitrary positive integer and $p\in\{1,\ldots,m\}$. Let $A$ be a non-empty set and let $B$ be a non-empty proper subset of $A$. Then the set $\left[B\subset A\right]^m_{p}\setminus \left[B\subset A\right]^m_{p-1}$ consists of all elements $(x_1,\ldots,x_m)$ of the power $A^m$ such that exactly $p$ coordinates of $(x_1,\ldots,x_m)$ belong to $A\setminus B$.
\end{lemma}

\begin{lemma}\label{lemma-4.2}
Let $m$ be an arbitrary positive integer and $p\in\{0, 1,\ldots,m\}$. Let $S$ be a semigroup, $A$  and  $B$ be ideals in $S$ such that $B\subset A$. Then $\left[B\subset A\right]^m_{p}$ is an ideal of the direct power $S^m$.
\end{lemma}

An subset $D$ of a semigroup $S$ is said to be $\omega$-unstable if $D$ is infinite and $aB\cup Ba\nsubseteq D$ for any $a\in D$ and any infinite subset $B\subseteq D$.

\begin{definition}[{\cite{GutikLawsonRepov2009}}]\label{definition-4.3}
An \emph{ideal series} (see, for example, \cite{CP}) for a semigroup $S$ is a chain
of ideals
\begin{equation*}
I_0\subseteq I_1 \subseteq I_2 \subseteq\cdots\subseteq I_n = S.
\end{equation*}
We call the ideal series \emph{tight} if $I_0$ is a finite set and $D_k=I_k\setminus I_{k-1}$ is an $\omega$-unstable subset for each $k=1,\ldots,n$.
\end{definition}

It is obvious that for every infinite cardinal $\lambda$ and any positive integer $n$ the semigroup $\mathscr{I}_\lambda^n$ has a tight ideal series.
A finite direct product of semigroups with tight ideal series is a semigroup with a tight ideal series and a homomorphic image of a semigroup with a tight ideal series with finite preimages is a semigroup with a tight ideal series too \cite{GutikLawsonRepov2009}.

\medskip

A subset $D$ of a semigroup $S$ is said to be \emph{strongly $\omega$-unstable} if $D$ is infinite and $aB\cup Bb\nsubseteq D$ for any $a,b\in D$ and any infinite subset $B\subseteq D$. It is obvious that a subset $D$ of a semigroup $S$ is strongly $\omega$-unstable then $D$ is $\omega$-unstable.

\begin{definition}\label{definition-4.5}
We call the ideal series
$
I_0\subseteq I_1 \subseteq I_2 \subseteq\cdots\subseteq I_n = S
$
\emph{strongly tight} if $I_0$ is a finite set and $D_k=I_k\setminus I_{k-1}$ is a strongly $\omega$-unstable subset for each $k=1,\ldots,n$.
\end{definition}

An example of a semigroup with a strongly tight ideal series gives the following proposition.

\begin{proposition}\label{proposition-4.6}
Let $\lambda$ be any infinite cardinal and $n$ be any positive integer. Then
\begin{equation*}
I_0=\{0\}\subseteq I_1=\mathscr{I}_\lambda^1 \subseteq I_2=\mathscr{I}_\lambda^2 \subseteq\cdots\subseteq I_n=\mathscr{I}_\lambda^n,
\end{equation*}
is the strongly tight ideal series in the semigroup $\mathscr{I}_\lambda^n$.
\end{proposition}

\begin{proof}
The definition of the semigroup $\mathscr{I}_\lambda^n$ implies that $I_0\subseteq I_1 \subseteq I_2 \subseteq\cdots\subseteq I_n$ is an ideal series in $\mathscr{I}_\lambda^n$.

Fix an arbitrary integer $i=1,\ldots,n$. For any infinite subset $B$ of $\mathscr{I}_\lambda^i\setminus \mathscr{I}_\lambda^{i-1}$ at least one of the following families of sets
\begin{equation*}
  \mathfrak{d}(B)=\left\{\operatorname{dom}\gamma\colon \gamma\in B\right\} \qquad \hbox{or} \qquad \mathfrak{r}(B)=\left\{\operatorname{ran}\gamma\colon \gamma\in B\right\}
\end{equation*}
is infinite.
Then the definition of the semigroup operation in $\mathscr{I}_\lambda^n$ implies that $\alpha B\nsubseteq\mathscr{I}_\lambda^i\setminus \mathscr{I}_\lambda^{i-1}$ in the case when the set $\mathfrak{d}(B)$ is infinite, and $B\beta\nsubseteq\mathscr{I}_\lambda^i\setminus \mathscr{I}_\lambda^{i-1}$ in the case when the set $\mathfrak{r}(B)$ is infinite, for any $\alpha,\beta\in\mathscr{I}_\lambda^i\setminus \mathscr{I}_\lambda^{i-1}$.
\end{proof}

Later for an arbitrary non-empty set $A$, any positive integer $n$ and any $i\in\{1,\ldots, n\}$ by $\pi_i\colon A^n\to A$, $(a_1,\ldots,a_n)\mapsto a_i$ we shall denote the projection on $i$-th factor of the power $A^n$.

\begin{proposition}\label{proposition-4.6a}
Let $n$ be a positive integer $\geqslant 2$ and let $I_0\subseteq I_1 \subseteq I_2 \subseteq\cdots\subseteq I_m = S$  be the strongly tight ideal series for a semigroup $S$. Then the following  series
\begin{equation}\label{eq-4.1}
\begin{split}
  I_0^n &\subseteq [I_0\subset I_1]^n_{1} \subseteq [I_0\subset I_1]^n_{2}\subseteq \cdots \subseteq [I_0\subset I_1]^n_{n-1} \subseteq [I_0\subset I_1]^n_{n}=I^n_1\subseteq\\
        & \subseteq  [I_1\subset I_2]^n_{1} \subseteq [I_1\subset I_2]^n_{2}\subseteq \cdots \subseteq [I_1\subset I_2]^n_{n-1} \subseteq [I_1\subset I_2]^n_{n}=I^n_2\subseteq\\
        & \subseteq  \qquad \cdots \qquad  \cdots \qquad  \subseteq\\
        & \subseteq  [I_{m-1}\subset I_m]^n_{1} \subseteq [I_{m-1}\subset I_m]^n_{2}\subseteq \cdots \subseteq [I_{m-1}\subset I_m]^n_{n-1} \subseteq [I_{m-1}\subset I_m]^n_{n}=I^n_m=S^n
\end{split}
\end{equation}
is a strongly tight ideal series for the direct power $S^n$.
\end{proposition}

\begin{proof}
It is obvious that $I_0^n$ is a finite ideal of $S^n$. Also by Lemma~\ref{lemma-4.2} all subsets in series  \eqref{eq-4.1} are ideals in $S^n$.

Fix any $k\in\{1,\ldots, m\}$ and any $p\in\{1,\ldots,n\}$. We claim that the sets
\begin{equation*}
[I_{k-1}\subset I_k]^n_{p}\setminus [I_{k-1}\subset I_k]^n_{p-1} \qquad \hbox{and} \qquad [I_{k-1}\subset I_k]^n_{1}\setminus I_{k-1}^n
\end{equation*}
are strongly $\omega$-unstable in $S^n$. Indeed, fix an arbitrary infinite subset $B\subseteq [I_{k-1}\subset I_k]^n_{p}\setminus [I_{k-1}\subset I_k]^n_{p-1}$ and any points $a=(a_1,\ldots,a_n), b=(b_1,\ldots,b_n)\in [I_{k-1}\subset I_k]^n_{p}\setminus [I_{k-1}\subset I_k]^n_{p-1}$. Then there exists a coordinate $i\in\{1,\ldots, n\}$ such that the set $\pi_i(B)\subseteq I_k\setminus I_{k-1}$ is infinite. If $a_i\notin I_k\setminus I_{k-1}$ or $b_i\notin I_k\setminus I_{k-1}$  then $(a_i\cdot \pi_i(B)\cup \pi_i(B)\cdot b_i)\cap I_k\setminus I_{k-1}=\varnothing$, and hence $aB\cup Bb\nsubseteq [I_{k-1}\subset I_k]^n_{p}\setminus [I_{k-1}\subset I_k]^n_{p-1}$. If $a_i,b_i\in I_k\setminus I_{k-1}$ then $(a_i\cdot \pi_i(B)\cup \pi_i(B)\cdot b_i)\nsubseteq I_k\setminus I_{k-1}$, because the set $I_k\setminus I_{k-1}$ is strongly $\omega$-unstable in $S$, and hence $aB\cup Bb\nsubseteq [I_{k-1}\subset I_k]^n_{p}\setminus [I_{k-1}\subset I_k]^n_{p-1}$. The proof of the statement that the set $[I_{k-1}\subset I_k]^n_{1}\setminus I_{k-1}^n$ is strongly  $\omega$-unstable in $S^n$ is similar.
\end{proof}

Later we fix an arbitrary positive integer $n$. Then for any semigroup $S$ and any positive integer $k\leqslant n$ since $\mathscr{I}_\lambda^k(S)$ is a subsemigroup of $\mathscr{I}_\lambda^n(S)$ by $\iota\colon\mathscr{I}_\lambda^k(S)\to\mathscr{I}_\lambda^n(S)$ we denote this natural embedding. Similar arguments imply that, without loss of generality for any subsemigroup $T$ of  $S$ and any positive integer $k\leqslant n$ since $\mathscr{I}_\lambda^k(T)$ is a subsemigroup of $\mathscr{I}_\lambda^n(S)$ by $\iota\colon\mathscr{I}_\lambda^k(T)\to\mathscr{I}_\lambda^n(S)$ we denote this natural embedding.

Let $A\neq\varnothing$ and $k$ be any positive integer. A subset $B\subseteq A^k$ is said to be \emph{$k$-symmetric} if the following condition holds: $(b_1,\ldots,b_k)\in B$ implies $\left(b_{(1)\sigma},\ldots,b_{(k)\sigma}\right)\in B$ for every permutation $\sigma\colon\{1,\ldots,k\}\to \{1,\ldots,k\}$.

\begin{remark}\label{remark-4.7}
We observe that every element of the tight ideal series \eqref{eq-4.1} is $m$-symmetric in $S^n$, and moreover the sets $[I_{k-1}\subset I_k]^n_{p}\setminus [I_{k-1}\subset I_k]^n_{p-1}$ and $[I_{k-1}\subset I_k]^n_{1}\setminus I_{k-1}^n$ is $m$-symmetric in $S^n$, too, for all $k\in\{1,\ldots, m\}$ and  $p\in\{1,\ldots,n\}$.
\end{remark}

We need the following construction.

\begin{construction}\label{construction-4.8}
Let $\lambda$ be cardinal $\geqslant 1$, $n$ be any positive integer, $k$ be any positive integer $\leqslant\min\{n,\lambda\}$ and $S$ be a semigroup. For any ordered collections of $k$ distinct elements  $(a_1,\ldots,a_k)$ and $(b_1,\ldots,b_k)$ of $\lambda^k$ we define a map $\mathfrak{f}^{(a_1,\ldots,a_k)}_{(b_1,\ldots,b_k)}\colon S^k\rightarrow S^{(a_1,\ldots,a_k)}_{(b_1,\ldots,b_k)}$ by the formula
\begin{equation*}
  (s_1,\ldots,s_k)\mathfrak{f}^{(a_1,\ldots,a_k)}_{(b_1,\ldots,b_k)}=
  \left(
    \begin{array}{ccc}
      a_1 & \cdots & a_k \\
      s_1 & \cdots & s_k \\
      b_1 & \cdots & b_k \\
    \end{array}
  \right).
\end{equation*}

For any non-empty subset $A\subseteq S^k$ and any positive integer $k\leqslant n$ we determine the following subsets
\begin{equation*}
  [A]^{(\ast)_k}_{\mathscr{I}_\lambda^n(S)}=\bigcup\left\{(A)\mathfrak{f}^{(a_1,\ldots,a_k)}_{(b_1,\ldots,b_k)}\colon (a_1,\ldots,a_k)\; \hbox{and} \; (b_1,\ldots,b_k) \; \hbox{are ordered collections of}\; k \; \hbox{distinct elements of}\; \lambda^k\right\}
\end{equation*}
and
\begin{equation*}
  \overline{[A]}^{(\ast)_k}_{\mathscr{I}_\lambda^n(S)}=
\left\{
  \begin{array}{cl}
    {[A]^{(\ast)_k}_{\mathscr{I}_{\lambda}^n(S)}\cup \mathscr{I}_\lambda^{k-1}(S)}, & \hbox{if~} k\geqslant 1; \\
    &\\
    {[A]^{(\ast)_1}_{\mathscr{I}_{\lambda}^n(S)}\cup\left\{0\right\}}, & \hbox{if~} k=1,
  \end{array}
\right.
\end{equation*}
of the semigroup $\mathscr{I}_\lambda^n(S)$.
\end{construction}

The following lemma follows from the definition of $k$-symmetric sets.

\begin{lemma}\label{lemma-4.9}
Let $\lambda$ be cardinal $\geqslant 1$, $k$ be any positive integer $\leqslant\lambda$ and $S$ be a semigroup. Let $(a_1,\ldots,a_k)$ and $(b_1,\ldots,b_k)$ be arbitrary ordered collections of $k$ distinct elements  of $\lambda^k$. If $A\neq\varnothing$ is a $k$-symmetric subset of $S^k$ then $(A)\mathfrak{f}^{(a_1,\ldots,a_k)}_{(b_1,\ldots,b_k)}=(A)\mathfrak{f}^{(a_{(1)\sigma},\ldots,a_{(k)\sigma})}_{(b_{(1)\sigma},\ldots,b_{(k)\sigma})}$ for every permutation $\sigma\colon\{1,\ldots,k\}\to \{1,\ldots,k\}$.
\end{lemma}

\begin{theorem}\label{theorem-4.10}
Let $\lambda$ be an infinite cardinal and $n$ be a positive integer. If $S$ is a finite semigroup then
  \begin{equation*}
I_0=\{0\}\subseteq I_1=\mathscr{I}_{\lambda}^{1}(S) \subseteq I_2=\mathscr{I}_{\lambda}^{2}(S) \subseteq\cdots\subseteq I_n=\mathscr{I}_{\lambda}^{n}(S)
\end{equation*}
is the strongly tight ideal series for the semigroup $\mathscr{I}_\lambda^{n}(S)$.
\end{theorem}

\begin{proof}
It is obvious that for every $i=0,1,\ldots, n$ the set $I_i$ is an ideal in $\mathscr{I}_\lambda^{n}(S)$ and moreover the set $I_0$ is finite.

Fix an arbitrary $i=1,\ldots, n$ and any infinite subset $B\subseteq I_i\setminus I_{i-1}$. Since the semigroup $S$ is finite, every infinite subset $X$ of $I_i\setminus I_{i-1}$ intersects infinitely many sets of the form $S^{(a_1,\ldots,a_i)}_{(b_1,\ldots,b_i)}$. Then the definition of the semigroup $\mathscr{I}_\lambda^{n}(S)$ implies that at least one of the following families of sets
\begin{equation*}
  \mathfrak{d}(B)=\left\{\operatorname{\mathbf{d}}\gamma\colon \gamma\in B\right\} \qquad \hbox{or} \qquad \mathfrak{r}(B)=\left\{\operatorname{\mathbf{r}}\gamma\colon \gamma\in B\right\}
\end{equation*}
is infinite.
Then the definition of the semigroup operation in $\mathscr{I}_\lambda^n(S)$ implies that $\alpha B\nsubseteq I_i\setminus I_{i-1}$ in the case when the set $\mathfrak{d}(B)$ is infinite, and $B\beta\nsubseteq I_i\setminus I_{i-1}$ in the case when the set $\mathfrak{r}(B)$ is infinite, for any $\alpha,\beta\in I_i\setminus I_{i-1}$.
\end{proof}

\begin{theorem}\label{theorem-4.11}
Let $\lambda$ be an infinite cardinal, $n$ be a positive integer and let $I_0\subseteq I_1 \subseteq I_2 \subseteq\cdots\subseteq I_m = S$  be the strongly tight ideal series for a  semigroup $S$. Then the following  series

\begin{equation}\label{eq-4.3}
\begin{split}
J_0=\{0\} & \subseteq J_{1,0}=\overline{[I_0]}^{(\ast)_1}_{\mathscr{I}_\lambda^n(S)} \subseteq\\
          & \subseteq J_{1,1}=\overline{[I_1]}^{(\ast)_1}_{\mathscr{I}_\lambda^n(S)} \subseteq
           J_{1,2}=\overline{[I_2]}^{(\ast)_1}_{\mathscr{I}_\lambda^n(S)}\subseteq \cdots \subseteq J_{1,m-1}=\overline{[I_{m-1}]}^{(\ast)_1}_{\mathscr{I}_\lambda^n(S)}\subseteq J_{1,m}=\overline{[I_{m}]}^{(\ast)_1}_{\mathscr{I}_\lambda^n(S)}=\mathscr{I}_\lambda^1(S)\subseteq\\
    & \subseteq J_{2,0}=\overline{[I_0^2]}^{(\ast)_2}_{\mathscr{I}_\lambda^n(S)} \subseteq\\
    & \subseteq J_{2,1}=\overline{\left[[I_0\subset I_1]^2_{1}\right]}^{(\ast)_2}_{\mathscr{I}_\lambda^n(S)} \subseteq
      J_{2,2}=\overline{\left[I_1^2\right]}^{(\ast)_2}_{\mathscr{I}_\lambda^n(S)} \subseteq \\
    & \subseteq J_{2,3}=\overline{\left[[I_1\subset I_2]^2_{1}\right]}^{(\ast)_2}_{\mathscr{I}_\lambda^n(S)} \subseteq
      J_{2,4}=\overline{\left[I_2^2\right]}^{(\ast)_2}_{\mathscr{I}_\lambda^n(S)} \subseteq \\
      & \subseteq  \qquad \cdots \qquad  \cdots \qquad  \subseteq\\
    & \subseteq J_{2,2m-1}=\overline{\left[[I_{m-1}\subset I_m]^2_{1}\right]}^{(\ast)_2}_{\mathscr{I}_\lambda^n(S)} \subseteq
      J_{2,2m}=\overline{\left[[I_m]^2_{2}\right]}^{(\ast)_2}_{\mathscr{I}_\lambda^n(S)}=\mathscr{I}_\lambda^2(S) \subseteq \\
   & \subseteq  \qquad \cdots \qquad  \cdots \qquad  \subseteq  \qquad \cdots \qquad  \cdots \qquad  \subseteq\\
   & \subseteq J_{n,0}=\overline{[I_0^n]}^{(\ast)_n}_{\mathscr{I}_\lambda^n(S)} \subseteq\\
   & \subseteq J_{n,1}=\overline{\left[[I_0\subset I_1]^n_{1}\right]}^{(\ast)_n}_{\mathscr{I}_\lambda^n(S)} \subseteq
      J_{n,2}=\overline{\left[[I_0\subset I_1]^n_{2}\right]}^{(\ast)_n}_{\mathscr{I}_\lambda^n(S)} \subseteq \\
   & \subseteq J_{n,3}=\overline{\left[[I_0\subset I_1]^n_{3}\right]}^{(\ast)_n}_{\mathscr{I}_\lambda^n(S)} \subseteq
      J_{n,4}=\overline{\left[[I_0\subset I_1]^n_{4}\right]}^{(\ast)_n}_{\mathscr{I}_\lambda^n(S)} \subseteq \\
   & \subseteq \cdots \subseteq \\
   & \subseteq J_{n,n-1}=\overline{\left[[I_0\subset I_1]^n_{n-1}\right]}^{(\ast)_n}_{\mathscr{I}_\lambda^n(S)} \subseteq
      J_{n,n}=\overline{\left[I_1^n\right]}^{(\ast)_n}_{\mathscr{I}_\lambda^n(S)} \subseteq \\
   & \subseteq J_{n,n+1}=\overline{\left[[I_1\subset I_2]^n_{1}\right]}^{(\ast)_n}_{\mathscr{I}_\lambda^n(S)} \subseteq
      J_{n,n+2}=\overline{\left[[I_1\subset I_2]^n_{2}\right]}^{(\ast)_n}_{\mathscr{I}_\lambda^n(S)} \subseteq \\
   & \subseteq J_{n,n+3}=\overline{\left[[I_1\subset I_2]^n_{3}\right]}^{(\ast)_n}_{\mathscr{I}_\lambda^n(S)} \subseteq
      J_{n,n+4}=\overline{\left[[I_1\subset I_2]^n_{4}\right]}^{(\ast)_n}_{\mathscr{I}_\lambda^n(S)} \subseteq \\
   & \subseteq \cdots \subseteq \\
   & \subseteq J_{n,2n-1}=\overline{\left[[I_1\subset I_2]^n_{n-1}\right]}^{(\ast)_n}_{\mathscr{I}_\lambda^n(S)} \subseteq
      J_{n,2n}=\overline{\left[I_2^n\right]}^{(\ast)_n}_{\mathscr{I}_\lambda^n(S)} \subseteq \\
            & \subseteq  \qquad \cdots \qquad  \cdots \qquad  \subseteq\\
   & \subseteq J_{n,(m-1)n+1}=\overline{\left[[I_{m-1}\subset I_m]^n_{1}\right]}^{(\ast)_n}_{\mathscr{I}_\lambda^n(S)} \subseteq
      J_{n,(m-1)n+2}=\overline{\left[[I_{m-1}\subset I_m]^n_{2}\right]}^{(\ast)_n}_{\mathscr{I}_\lambda^n(S)} \subseteq \\
   & \subseteq J_{n,(m-1)n+3}=\overline{\left[[I_{m-1}\subset I_m]^n_{3}\right]}^{(\ast)_n}_{\mathscr{I}_\lambda^n(S)} \subseteq
      J_{n,(m-1)n+4}=\overline{\left[[I_{m-1}\subset I_m]^n_{4}\right]}^{(\ast)_n}_{\mathscr{I}_\lambda^n(S)} \subseteq \\
   & \subseteq \cdots \subseteq \\
   & \subseteq J_{n,mn-1}=\overline{\left[[I_{m-1}\subset I_m]^n_{n-1}\right]}^{(\ast)_n}_{\mathscr{I}_\lambda^n(S)} \subseteq
      J_{n,mn}=\overline{\left[I_m^n\right]}^{(\ast)_n}_{\mathscr{I}_\lambda^n(S)} =\mathscr{I}_\lambda^{n}(S)
\end{split}
\end{equation}
is a strongly tight ideal series for the semigroup $\mathscr{I}_\lambda^{n}(S)$.
\end{theorem}

\begin{proof}
The definition of the semigroup $\mathscr{I}_\lambda^{n}(S)$ and Lemma~\ref{lemma-4.2} imply that all subsets in series  \eqref{eq-4.3} are ideals in $\mathscr{I}_\lambda^{n}(S)$.

Since $I_0$ is a finite ideal in $S$, the following equalities

\begin{align*}
  J_{1,0}\setminus J_0 & =\overline{[I_0]}^{(\ast)_1}_{\mathscr{I}_\lambda^n(S)}\setminus \{0\}=[I_0]^{(\ast)_1}_{\mathscr{I}_\lambda^n(S)} \\
   J_{2,0}\setminus J_{1,m} & =\overline{[I_0^2]}^{(\ast)_1}_{\mathscr{I}_\lambda^n(S)}\setminus \mathscr{I}_\lambda^1(S)= [I_0^2]^{(\ast)_1}_{\mathscr{I}_\lambda^n(S)} \\
   & \cdots \qquad \cdots \qquad \cdots\\
   J_{n,0}\setminus J_{n-1,m(n-1)} & =\overline{[I_0^n]}^{(\ast)_n}_{\mathscr{I}_\lambda^{n}(S)}\setminus \mathscr{I}_\lambda^{n-1}(S)= [I_0^n]^{(\ast)_n}_{\mathscr{I}_\lambda^n(S)}
\end{align*}
and the semigroup operation of $\mathscr{I}_\lambda^n(S)$ imply that
\begin{equation*}
  J_{1,0}\setminus J_0, \quad J_{2,0}\setminus J_{1,m}, \quad \ldots , \quad   J_{n,0}\setminus J_{n-1,m(n-1)}
\end{equation*}
are strongly $\omega$-unstable subsets in $\mathscr{I}_\lambda^n(S)$.

Next we shall show that the set $J_{k,p}\setminus J_{k,p-1}$ is strongly $\omega$-unstable in $\mathscr{I}_\lambda^n(S)$ for all $k=1,\ldots,n$ and $p=1,\ldots,km$.

Fix any infinite subset $B$ of $J_{k,p}\setminus J_{k,p-1}$ and any $\alpha,\beta\in J_{k,p}\setminus J_{k,p-1}$. If $\mathop{\textsf{\textbf{d}}}(B)\neq\mathop{\textsf{\textbf{r}}}(\alpha)$ then the semigroup operation of $\mathscr{I}_\lambda^n(S)$ implies that $\alpha B\nsubseteq J_{k,p}\setminus J_{k,p-1}$. Similar, if $\mathop{\textsf{\textbf{d}}}(\beta)\neq\mathop{\textsf{\textbf{r}}}(B)$ then $B\beta\nsubseteq J_{k,p}\setminus J_{k,p-1}$.

Next we suppose that $\mathop{\textsf{\textbf{d}}}(B)=\mathop{\textsf{\textbf{r}}}(\alpha)$, $\mathop{\textsf{\textbf{d}}}(\beta)=\mathop{\textsf{\textbf{r}}}(B)$,
\begin{equation*}
\alpha=
\left(
  \begin{array}{ccc}
    a_1 & \cdots & a_k \\
    s_1 & \cdots & s_k \\
    b_1 & \cdots & b_k \\
  \end{array}
\right)
\qquad \hbox{and} \qquad
\beta=
\left(
  \begin{array}{ccc}
    c_1 & \cdots & c_k \\
    t_1 & \cdots & t_k \\
    d_1 & \cdots & d_k \\
  \end{array}
\right),
\end{equation*}
for some $s_1, \ldots, s_k, t_1, \ldots, t_k \in S$ and ordered collections of $k$ distinct elements $(a_1,\ldots,a_k)$, $(b_1,\ldots,b_k)$, $(c_1,\ldots,c_k)$, $(d_1,\ldots,d_k)$ of $\lambda^k$. Then the set $B$ consists of elements of the form
\begin{equation*}
\gamma=
\left(
  \begin{array}{ccc}
    b_1 & \cdots & b_k \\
    x_1 & \cdots & x_k \\
    c_{(1)\sigma} & \cdots & c_{(k)\sigma} \\
  \end{array}
\right),
\end{equation*}
where $x_1, \ldots, x_k \in S$ and $\sigma\colon \{1,\ldots,k\}\to \{1,\ldots,k\}$ is a permutation.

First we consider the case when $J_{k,p}=J_{k,jk}=\overline{\left[I_j^k\right]}^{(\ast)_k}_{\mathscr{I}_\lambda^n(S)}$ for some $j=1,\ldots,m$. Then
\begin{equation*}
J_{k,p-1}=J_{k,jk-1}=\overline{\left[[I_{j-1}\subset I_j]^k_{k-1}\right]}^{(\ast)_k}_{\mathscr{I}_\lambda^n(S)}
\end{equation*}
and $B\subseteq \left[I_j^k\right]^{(\ast)_k}_{\mathscr{I}_\lambda^n(S)}$. Since the set $B$ is infinite, there exists $b_{i_0}\in\{b_1,\ldots,b_k\}$ such that there exist infinitely many $\gamma\in B$ such that $\mathbf{d}(\gamma)\ni b_{i_0}$. Without loss of generality we may assume that $b_{i_0}=b_{1}$. We put $B_0=\left\{\gamma\in B\colon b_{1}\in\mathbf{d}(\gamma)\right\}$. Then the set $B_0$ is infinite and hence the set
\begin{equation*}
B_0^S=\left\{x_1\in S\colon
\left(
  \begin{array}{ccc}
    b_1 & \cdots & b_k \\
    x_1 & \cdots & x_k \\
    c_{(1)\sigma} & \cdots & c_{(k)\sigma} \\
  \end{array}
\right)\in B_0, \; \sigma \hbox{~is a permutation of~} \{1,\ldots,k\}
\right\}
\end{equation*}
is infinite, too. The above implies that there exists a permutation $\sigma_0$ of $\{1,\ldots,k\}$ such that infinitely many elements of the form
$
\left(
  \begin{array}{ccc}
    b_1 & \cdots & b_k \\
    x_1 & \cdots & x_k \\
    c_{(1)\sigma_0} & \cdots & c_{(k)\sigma_0} \\
  \end{array}
\right)
$
belong to $B_0$. Since $s_1,t_{(1)\sigma_0}\in I_j\setminus I_{j-1}$ and the set $I_j\setminus I_{j-1}$ is strongly $\omega$-unstable we obtain that  $a_1\cdot B_0^S\cup B_0^S\cdot t_{(1)\sigma_0}\nsubseteq I_j\setminus I_{j-1}$, and hence the set $\left[I_j^k\right]^{(\ast)_k}_{\mathscr{I}_\lambda^n(S)}$ is strongly $\omega$-unstable, as well.

Next we consider the case $J_{k,p}=J_{n,(j-1)k+q}=\overline{\left[[I_{j-1}\subset I_j]^k_{q}\right]}^{(\ast)_k}_{\mathscr{I}_\lambda^n(S)}$ for some $j=1,\ldots,m$. Then
\begin{equation*}
J_{k,p-1}=J_{n,(j-1)k+q-1}=\overline{\left[[I_{j-1}\subset I_j]^k_{q-1}\right]}^{(\ast)_k}_{\mathscr{I}_\lambda^n(S)}
\end{equation*}
and $B\subseteq \left[[I_{j-1}\subset I_j]^k_{q}\right]^{(\ast)_k}_{\mathscr{I}_\lambda^n(S)}$. Since the set $B$ is infinite, without loss of generality we may assume that $B$ contains an infinite subset $B_0$ which consists of elements of the form
\begin{equation}\label{eq-4.4}
  \gamma=
\left(
  \begin{array}{cccccc}
    b_1 & \cdots & b_q & b_{q+1} & \cdots & b_k\\
    x_1 & \cdots & x_q & x_{q+1} & \cdots & s_k\\
    c_1 & \cdots & c_q & c_{q+1} & \cdots & c_k\\
  \end{array}
\right),
\end{equation}
where $x_1,\ldots,x_q\in I_j\setminus I_{j-1}$ and $x_{q+1},\ldots,x_k\in I_{j-1}\setminus I_{j-2}$ for some ordered collections of $k$ distinct elements $(b_1,\ldots,b_k)$ and $(c_1,\ldots,c_k)$ of $\lambda^k$. Fix arbitrary elements
\begin{equation*}
\alpha=
\left(
  \begin{array}{ccc}
    a_1 & \cdots & a_k \\
    s_1 & \cdots & s_k \\
    b_1 & \cdots & b_k \\
  \end{array}
\right)
\qquad \hbox{and} \qquad
\beta=
\left(
  \begin{array}{ccc}
    c_1 & \cdots & c_k \\
    t_1 & \cdots & t_k \\
    d_1 & \cdots & d_k \\
  \end{array}
\right),
\end{equation*}
of the set $B$. If $s_u\notin I_j\setminus I_{j-1}$ for some $u\in \{1,\ldots,q\}$ or $s_v\notin I_{j-1}\setminus I_{j-2}$ for some $v\in \{q+1,\ldots,k\}$ then $\alpha B_0\nsubseteq \left[[I_{j-1}\subset I_j]^k_{q}\right]^{(\ast)_k}_{\mathscr{I}_\lambda^n(S)}$. Similarly, $t_u\notin I_j\setminus I_{j-1}$ for some $u\in \{1,\ldots,q\}$ or $t_v\notin I_{j-1}\setminus I_{j-2}$ for some $v\in \{q+1,\ldots,k\}$ then $B_0\beta\nsubseteq \left[[I_{j-1}\subset I_j]^k_{q}\right]^{(\ast)_k}_{\mathscr{I}_\lambda^n(S)}$. Hence later we shall assume that $s_u\in I_j\setminus I_{j-1}$ for all $u\in \{1,\ldots,q\}$, $s_v\in I_{j-1}\setminus I_{j-2}$ for all $v\in \{q+1,\ldots,k\}$, $t_u\in I_j\setminus I_{j-1}$ for all $u\in \{1,\ldots,q\}$ and $t_v\in I_{j-1}\setminus I_{j-2}$ for all $v\in \{q+1,\ldots,k\}$. Since the set $B_0$ is infinite there exists ${i_0}\in\{1,\ldots,k\}$ such that there exist infinitely many $\gamma\in B_0$ such that $\mathbf{d}(\gamma)\ni b_{i_0}$. We put $B_1=\left\{\gamma\in B_0\colon b_{i_0}\in\mathbf{d}(\gamma)\right\}$. Since the set $B_1$ is inifinite the following statements hold:
\begin{itemize}
  \item[$(1)$] if $i_0\in\{1,\ldots,q\}$ then $s_{i_0}A\cup At_{i_0}\nsubseteq I_j\setminus I_{j-1}$, where
  \begin{equation*}
    A=\left\{x_{i_0}\colon
    \gamma=
\left(
  \begin{array}{ccccccc}
    b_1 & \cdots & b_{i_0} & \cdots & b_{q} & \cdots & b_k\\
    x_1 & \cdots & x_{i_0} & \cdots & x_{q} & \cdots & s_k\\
    c_1 & \cdots & c_{i_0} & \cdots & c_{q} & \cdots & c_k\\
  \end{array}
\right)\in B_1
    \right\},
  \end{equation*}
  because the set $I_j\setminus I_{j-1}$ is strongly $\omega$-unstable in $S$;
  \item[$(2)$] if $i_0\in\{q+1,\ldots,k\}$ then $s_{i_0}A\cup At_{i_0}\nsubseteq I_{j-1}\setminus I_{j-2}$, where
  \begin{equation*}
    A=\left\{x_{i_0}\colon
    \gamma=
\left(
  \begin{array}{ccccccc}
    b_1 & \cdots & b_{q} & \cdots & b_{i_0} & \cdots & b_k\\
    x_1 & \cdots & x_{q} & \cdots & x_{i_0} & \cdots & s_k\\
    c_1 & \cdots & c_{q} & \cdots & c_{i_0} & \cdots & c_k\\
  \end{array}
\right)\in B_1
    \right\},
  \end{equation*}
  because the set $I_{j-1}\setminus I_{j-2}$ is strongly $\omega$-unstable in $S$.
\end{itemize}

Both above statements imply that $\alpha B_1\cup B_1\gamma\nsubseteq \left[[I_{j-1}\subset I_j]^k_{q}\right]^{(\ast)_k}_{\mathscr{I}_\lambda^n(S)}$ and hence $\alpha B\cup B_1\gamma\nsubseteq \left[[I_{j-1}\subset I_j]^k_{q}\right]^{(\ast)_k}_{\mathscr{I}_\lambda^n(S)}$, i.e., the set $\left[[I_{j-1}\subset I_j]^k_{q}\right]^{(\ast)_k}_{\mathscr{I}_\lambda^n(S)}$ is strongly $\omega$-unstable in $\mathscr{I}_\lambda^{n}(S)$. This completed the proof of the theorem.
\end{proof}

Theorem~\ref{theorem-4.11} implies the following

\begin{corollary}\label{corollary-4.12}
Let $\lambda$ be an infinite cardinal, $n$ be a positive integer and let $I_0\subseteq I_1 \subseteq I_2 \subseteq\cdots\subseteq I_m = S$  be the strongly tight ideal series for a  semigroup $S$. Then the ideal series \eqref{eq-4.3} is tight for the semigroup $\mathscr{I}_\lambda^{n}(S)$.
\end{corollary}

The proof of the following theorem is similar to Theorem~\ref{theorem-4.11}.

\begin{theorem}\label{theorem-4.13}
Let $\lambda$ be a finite cardinal, $n$ be a positive integer $\leqslant \lambda$ and let $I_0\subseteq I_1 \subseteq I_2 \subseteq\cdots\subseteq I_m = S$  be the strongly tight ideal series for a  semigroup $S$. Then the following  series

\begin{equation}\label{eq-4.5}
\begin{split}
J_0&=\{0\}\cup\overline{[I_0]}^{(\ast)_1}_{\mathscr{I}_\lambda^n(S)} \subseteq \\
       & \subseteq J_{1,1}=\overline{[I_1]}^{(\ast)_1}_{\mathscr{I}_\lambda^n(S)} \subseteq
           J_{1,2}=\overline{[I_2]}^{(\ast)_1}_{\mathscr{I}_\lambda^n(S)}\subseteq \cdots \subseteq J_{1,m-1}=\overline{[I_{m-1}]}^{(\ast)_1}_{\mathscr{I}_\lambda^n(S)}\subseteq J_{1,m}=\overline{[I_{m}]}^{(\ast)_1}_{\mathscr{I}_\lambda^n(S)}=\mathscr{I}_\lambda^1(S)\subseteq\\
    & \subseteq J_{2,1}=\overline{\left[[I_1\subset I_2]^2_{1}\right]}^{(\ast)_2}_{\mathscr{I}_\lambda^n(S)} \subseteq
      J_{2,2}=\overline{\left[[I_1\subset I_2]^2_{2}\right]}^{(\ast)_2}_{\mathscr{I}_\lambda^n(S)} \subseteq \\
      & \subseteq  \qquad \cdots \qquad  \cdots \qquad  \subseteq\\
    & \subseteq J_{2,2m-1}=\overline{\left[[I_{m-1}\subset I_m]^2_{1}\right]}^{(\ast)_2}_{\mathscr{I}_\lambda^n(S)} \subseteq
      J_{2,2m}=\overline{\left[[I_m]^2_{2}\right]}^{(\ast)_2}_{\mathscr{I}_\lambda^n(S)}\mathscr{I}_\lambda^2(S) \subseteq \\
   & \subseteq  \qquad \cdots \qquad  \cdots \qquad  \subseteq  \qquad \cdots \qquad  \cdots \qquad  \subseteq\\
   & \subseteq J_{n,1}=\overline{\left[[I_0\subset I_1]^n_{1}\right]}^{(\ast)_n}_{\mathscr{I}_\lambda^n(S)} \subseteq
      J_{n,2}=\overline{\left[[I_0\subset I_1]^n_{2}\right]}^{(\ast)_n}_{\mathscr{I}_\lambda^n(S)} \subseteq \\
   & \subseteq J_{n,3}=\overline{\left[[I_0\subset I_1]^n_{3}\right]}^{(\ast)_n}_{\mathscr{I}_\lambda^n(S)} \subseteq
      J_{n,4}=\overline{\left[[I_0\subset I_1]^n_{4}\right]}^{(\ast)_n}_{\mathscr{I}_\lambda^n(S)} \subseteq \\
   & \subseteq \cdots \subseteq \\
   & \subseteq J_{n,n-1}=\overline{\left[[I_0\subset I_1]^n_{n-1}\right]}^{(\ast)_n}_{\mathscr{I}_\lambda^n(S)} \subseteq
      J_{n,n}=\overline{\left[I_1^n\right]}^{(\ast)_n}_{\mathscr{I}_\lambda^n(S)} \subseteq \\
   & \subseteq J_{n,n+1}=\overline{\left[[I_1\subset I_2]^n_{1}\right]}^{(\ast)_n}_{\mathscr{I}_\lambda^n(S)} \subseteq
      J_{n,n+2}=\overline{\left[[I_1\subset I_2]^n_{2}\right]}^{(\ast)_n}_{\mathscr{I}_\lambda^n(S)} \subseteq \\
   & \subseteq J_{n,n+3}=\overline{\left[[I_1\subset I_2]^n_{3}\right]}^{(\ast)_n}_{\mathscr{I}_\lambda^n(S)} \subseteq
      J_{n,n+4}=\overline{\left[[I_1\subset I_2]^n_{4}\right]}^{(\ast)_n}_{\mathscr{I}_\lambda^n(S)} \subseteq \\
   & \subseteq \cdots \subseteq \\
   & \subseteq J_{n,2n-1}=\overline{\left[[I_1\subset I_2]^n_{n-1}\right]}^{(\ast)_n}_{\mathscr{I}_\lambda^n(S)} \subseteq
      J_{n,2n}=\overline{\left[I_2^n\right]}^{(\ast)_n}_{\mathscr{I}_\lambda^n(S)} \subseteq \\
            & \subseteq  \qquad \cdots \qquad  \cdots \qquad  \subseteq\\
   & \subseteq J_{n,(m-1)n+1}=\overline{\left[[I_{m-1}\subset I_m]^n_{1}\right]}^{(\ast)_n}_{\mathscr{I}_\lambda^n(S)} \subseteq
      J_{n,(m-1)n+2}=\overline{\left[[I_{m-1}\subset I_m]^n_{2}\right]}^{(\ast)_n}_{\mathscr{I}_\lambda^n(S)} \subseteq \\
   & \subseteq J_{n,(m-1)n+3}=\overline{\left[[I_{m-1}\subset I_m]^n_{3}\right]}^{(\ast)_n}_{\mathscr{I}_\lambda^n(S)} \subseteq
      J_{n,(m-1)n+4}=\overline{\left[[I_{m-1}\subset I_m]^n_{4}\right]}^{(\ast)_n}_{\mathscr{I}_\lambda^n(S)} \subseteq \\
   & \subseteq \cdots \subseteq \\
   & \subseteq J_{n,mn-1}=\overline{\left[[I_{m-1}\subset I_m]^n_{n-1}\right]}^{(\ast)_n}_{\mathscr{I}_\lambda^n(S)} \subseteq
      J_{n,mn}=\overline{\left[I_m^n\right]}^{(\ast)_n}_{\mathscr{I}_\lambda^n(S)} =\mathscr{I}_\lambda^{n}(S)
\end{split}
\end{equation}
is a strongly tight ideal series for the semigroup $\mathscr{I}_\lambda^{n}(S)$.
\end{theorem}

Theorem~\ref{theorem-4.13} implies the following

\begin{corollary}\label{corollary-4.14}
Let $\lambda$ be a finite cardinal, $n$ be a positive integer $\leqslant \lambda$ and let $I_0\subseteq I_1 \subseteq I_2 \subseteq\cdots\subseteq I_m = S$  be the strongly tight ideal series for a  semigroup $S$.  Then the ideal series \eqref{eq-4.3} is tight for the semigroup $\mathscr{I}_\lambda^{n}(S)$.
\end{corollary}

Let $\mathfrak{S}$ be a class of semitopological semigroups. A semigroup $S\in\mathfrak{S}$ is called {\it $H$-closed in} $\mathfrak{S}$, if $S$ is a closed subsemigroup of any semitopological semigroup $T\in\mathfrak{S}$ which contains $S$ both as a subsemigroup and as a topological space. The $H$-closed topological semigroups were introduced by Stepp in \cite{Stepp1969}, and there they were called {\it maximal semigroups}.  An algebraic semigroup $S$ is called:{\it algebraically complete in} $\mathfrak{S}$, if $S$ with any Hausdorff topology $\tau$ such that $(S,\tau)\in\mathfrak{S}$ is $H$-closed in $\mathfrak{S}$. We observe that many distinct types of $H$-closedness of topological and semitopological semigroups studied in \cite{Banakh-Bardyla-2019}--\cite{Chuchman-Gutik-2007}, \cite{Gutik2014}--\cite{Gutik-Pavlyk-2005}, \cite{GutikReiter2009}, \cite{GutikRepovs-2008}.

By Proposition~10 from \cite{GutikLawsonRepov2009} every inverse semigroup $S$ with a tight ideal series is algebraically complete in the class of Hausdorff semitopological inverse semigroups with continuous inversion. Hence Proposition~\ref{proposition-2.5} and Theorems~\ref{theorem-4.11}, \ref{theorem-4.13} imply the following

\begin{theorem}\label{theorem-4.15}
Let $S$ be an inverse semigroup which admits a strongly tight ideal series. Then for every non-zero cardinal $\lambda$ and any positive integer $n\leqslant\lambda$ the semigroup $\mathscr{I}_\lambda^{n}(S)$ is algebraically complete in the class of Hausdorff semitopological inverse semigroups with continuous inversion.
\end{theorem}

We remark that in the case when $n=1$ the construction of $\mathscr{I}_\lambda^{1}(S)$ preserves the property to be a semigroup with a tight ideal series, and this follows from the following theorem.

\begin{theorem}\label{theorem-4.16}
Let $\lambda$ be any non-zero cardinal, $n$ be a positive integer $n\leqslant \lambda$ and let $I_0\subseteq I_1 \subseteq I_2 \subseteq\cdots\subseteq I_m = S$  be the tight ideal series for a  semigroup $S$. Then the following  series
\begin{equation}\label{eq-4.6}
J_0=\{0\}  \subseteq J_{1}=\overline{[I_0]}^{(\ast)_1}_{\mathscr{I}_\lambda^n(S)}
           \subseteq J_{2}=\overline{[I_1]}^{(\ast)_1}_{\mathscr{I}_\lambda^n(S)} \subseteq
           J_{3}=\overline{[I_2]}^{(\ast)_1}_{\mathscr{I}_\lambda^n(S)}\subseteq \cdots \subseteq J_{m}=\overline{[I_{m-1}]}^{(\ast)_1}_{\mathscr{I}_\lambda^n(S)}\subseteq J_{m+1}=\mathscr{I}_\lambda^1(S)
\end{equation}
is a tight ideal series for the semigroup $\mathscr{I}_\lambda^{1}(S)$ in the case when $\lambda$ is infinite, and
\begin{equation}\label{eq-4.7}
J_0=\{0\}\cup\overline{[I_0]}^{(\ast)_1}_{\mathscr{I}_\lambda^n(S)}
        \subseteq J_{1}=\overline{[I_1]}^{(\ast)_1}_{\mathscr{I}_\lambda^n(S)} \subseteq
           J_{2}=\overline{[I_2]}^{(\ast)_1}_{\mathscr{I}_\lambda^n(S)}\subseteq \cdots \subseteq J_{m-1}=\overline{[I_{m-1}]}^{(\ast)_1}_{\mathscr{I}_\lambda^n(S)}\subseteq J_{m}=\mathscr{I}_\lambda^1(S)
\end{equation}
is a tight ideal series for the semigroup $\mathscr{I}_\lambda^{1}(S)$ in the case when $\lambda$ is finite.
\end{theorem}

\begin{proof}
We consider the case when cardinal $\lambda$ is infinite. In the other case the proof is similar.

The semigroup operation of $\mathscr{I}_\lambda^{1}(S)$ implies that the the set $J_k$ is ideal in $\mathscr{I}_\lambda^{1}(S)$ for an arbitrary integer $k\in \{0,1,\ldots,m+1\}$.

Fix an arbitrary $k\in \{1,\ldots,m+1\}$. Then for any infinite subset $B$ of $J_{k}\setminus J_{k-1}$ and any $\alpha=
\left(
  \begin{array}{c}
    a \\
    s \\
    b \\
  \end{array}
\right)
\in J_{k}\setminus J_{k-1}$ the following statements hold.
\begin{itemize}
  \item[$(1)$] If $B\cap S^{(i)}_{(i)}$ is infinite for some $i\in\lambda$ then $B\cap S^{(i)}_{(i)}\subseteq [I_{k-1}\setminus I_{k_2}]^{(i)}_{(i)}$. Hence, the semigroup operation of $\mathscr{I}_\lambda^{1}(S)$ implies that $\alpha B\cup B\alpha\nsubseteq J_{k}\setminus J_{k-1}$ in the case when $a=b=i$, because the set $I_{k-1}\setminus I_{k_2}$ is $\omega$-unstable in $S$. Otherwise $0\in \alpha B\cup B\alpha\nsubseteq J_{k}\setminus J_{k-1}$.
  \item[$(2)$] In the other case the semigroup operation of $\mathscr{I}_\lambda^{1}(S)$ implies that $0\in\alpha B\cup B\alpha\nsubseteq J_{k}\setminus J_{k-1}$.
\end{itemize}
Both above statements imply that the set $J_{k}\setminus J_{k-1}$ is $\omega$-unstable in $\mathscr{I}_\lambda^{1}(S)$, which completes the proof of the theorem.
\end{proof}

\section{On a semitopological semigroup $\mathscr{I}_\lambda^n(S)$}

For any element $\alpha=
\left(
  \begin{array}{ccc}
    i_1 & \ldots & i_k \\
    j_1 & \ldots & j_k \\
  \end{array}
\right)
$
of the semigroup  $\mathscr{I}_\lambda^n$ and any $s\in S$ we denote
$\alpha[s]=
\left(
  \begin{array}{ccc}
    i_1 & \ldots & i_k \\
    s & \ldots & s \\
    j_1 & \ldots & j_k \\
  \end{array}
\right)
$
which is the element of $\mathscr{I}_\lambda^n(S)$. Later in this case we shall say that $\alpha[s]$ is the \emph{$s$-extension} of $\alpha$ or $\alpha$ is the \emph{$\mathscr{I}_\lambda^n$-restriction} of $\alpha[s]$.

\begin{proposition}\label{proposition-5.1}
Let $S$ be a monoid, $\lambda$ be any non-zero cardinal, $n$ be an arbitrary positive integer $\leqslant\lambda$, $0<k\leqslant n$ and $\mathscr{I}_\lambda^n(S)$ be a Hausdorff semitopological semigroup. Then for any ordered collections of $k$ distinct elements $(a_1,\ldots,a_k)$ and $(b_1,\ldots,b_k)$ of $\lambda^k$ and any element $\alpha_S\in S^{(a_1,\ldots,a_k)}_{(b_1,\ldots,b_k)}$ there exists an open neighbourhood $U(\alpha_S)$ of $\alpha_S$ such that
\begin{itemize}
  \item $U(\alpha_S)\cap\mathscr{I}_\lambda^{k-1}(S)=\varnothing$ and $U(\alpha_S)\cap\mathscr{I}_\lambda^{k}(S)\subseteq S^{(a_1,\ldots,a_k)}_{(b_1,\ldots,b_k)}$ in the case when $k\geqslant 2$,
  \item $0\notin U(\alpha_S)$ and $U(\alpha_S)\cap\mathscr{I}_\lambda^{1}(S)\subseteq S^{(a_1)}_{(b_1)}$ in the case when $k=1$.
\end{itemize}
Thus $\mathscr{I}_\lambda^{k}(S)$ is a closed subsemigroup of $\mathscr{I}_\lambda^{n}(S)$.
\end{proposition}

\begin{proof}
Fix an arbitrary $k\leqslant n$ and an arbitrary
$\alpha_S=\left(
\begin{array}{ccc}
a_1 & \ldots & a_k \\
s_1 & \ldots & s_k \\
b_1 & \ldots & b_k \\
\end{array}
\right)
\in S^{a_1,\ldots,a_k}_{b_1,\ldots,b_k}$.
It is obvious that $\varepsilon_{1}[1_S]\cdot \alpha_S\cdot\varepsilon_{2}[1_S]=\alpha_S$, where
\begin{equation*}
\varepsilon_{1}[1_S]=\left(
\begin{array}{ccc}
a_1 & \ldots & a_k \\
1_S & \ldots & 1_S \\
a_1 & \ldots & a_k \\
\end{array}
\right), \qquad
\varepsilon_{2}[1_S]=
\left(
\begin{array}{ccc}
b_1 & \ldots & b_k \\
1_S & \ldots & 1_S \\
b_1 & \ldots & b_k \\
\end{array}
\right),
\end{equation*}
and $1_S$ is the unit element of $S$.

Simple calculation implies that
\begin{equation*}
  S^{(a_1,\ldots,a_k)}_{(b_1,\ldots,b_k)}=\varepsilon_{1}[1_S]\cdot \mathscr{I}_\lambda^{n}(S)\cdot \varepsilon_{2}[1_S]\setminus \bigcup\left\{ \overline{\varepsilon}_{1}[1_S]\cdot \mathscr{I}_\lambda^{n}(S)\cdot \overline{\varepsilon}_{2}[1_S]\colon \overline{\varepsilon}_{1}<\varepsilon_{1} \hbox{~and~} \overline{\varepsilon}_{2}<\varepsilon_{2} \hbox{~in~} E(\mathscr{I}_\lambda^n)\right\}
\end{equation*}

We observe that $eT$ and $Te$ are closed subset in an arbitrary Hausdorff semitopological semigroup $T$ for any its idempotent $e$. Since for any idempotent $\varepsilon\in \mathscr{I}_\lambda^n$ the set ${\downarrow}\varepsilon=\{\iota\in E(\mathscr{I}_\lambda^n)\colon \iota\leqslant \varepsilon\}$ is finite, the set
\begin{equation*}
  A_{\alpha_S}=\bigcup\left\{ \overline{\varepsilon}_{1}[1_S]\cdot \mathscr{I}_\lambda^{n}(S)\cdot \overline{\varepsilon}_{2}[1_S]\colon \overline{\varepsilon}_{1}<\varepsilon_{1} \hbox{~and~} \overline{\varepsilon}_{2}<\varepsilon_{2}\right\}
\end{equation*}
is closed in $\mathscr{I}_\lambda^{n}(S)$. Fix an arbitrary open neighbourhood $W(\alpha_S)$ of $\alpha_S$ such that $W(\alpha_S)\cap A_{\alpha_S}=\varnothing$. The separate continuity of the semigroup operation on $\mathscr{I}_\lambda^{n}(S)$ implies that there exist an open neighbourhood $U(\alpha_S)$ of $\alpha_S$ such that $\varepsilon_{1}[1_S]\cdot U(\alpha_S)\cdot\varepsilon_{2}[1_S]\subseteq W(\alpha_S)$. The neighbourhood $U(\alpha_S)$ is requested. Indeed, if there exists $\beta_S\in \mathscr{I}_\lambda^{k}(S)\setminus S^{(a_1,\ldots,a_k)}_{(b_1,\ldots,b_k)}$ then $\varepsilon_{1}[1_S]\cdot \beta_S\cdot\varepsilon_{2}[1_S]\in A_{\alpha_S}$.
\end{proof}

\begin{remark}\label{remark-5.2}
We observe that in Proposition~\ref{proposition-5.1} we may assume that the neighbourhood  $U(\alpha_S)$ may be chosen with the following property $\varepsilon_{1}[1_S]\cdot U(\alpha_S)\cdot\varepsilon_{2}[1_S] \subseteq S^{(a_1,\ldots,a_k)}_{(b_1,\ldots,b_k)}$.
\end{remark}

\begin{proposition}\label{proposition-5.3}
Let $S$ be a monoid, $\lambda$ be any non-zero cardinal, $n$ be an arbitrary positive integer $\leqslant\lambda$, $0<k\leqslant n$ and $\mathscr{I}_\lambda^n(S)$ be a Hausdorff semitopological semigroup. Then for any  ordered collections of $k$ distinct elements $(a_1,\ldots,a_k)$, $(b_1,\ldots,b_k)$, $(c_1,\ldots,c_k)$, and $(d_1,\ldots,d_k)$ of $\lambda^k$ the subspaces $S^{(a_1,\ldots,a_k)}_{(b_1,\ldots,b_k)}$ and $S^{(c_1,\ldots,c_k)}_{(d_1,\ldots,d_k)}$ are homeomorphic, and moreover $S^{(a_1,\ldots,a_k)}_{(a_1,\ldots,a_k)}$ and $S^{(c_1,\ldots,c_k)}_{(c_1,\ldots,c_k)}$ are topologically isomorphic subsemigroups of $\mathscr{I}_\lambda^n(S)$.
\end{proposition}

\begin{proof}
Since $\mathscr{I}_\lambda^n(S)$ is a semitopological semigroup, the restrictions of the following maps
\begin{equation*}
  ^{(a_1,\ldots,a_k)}_{(b_1,\ldots,b_k)}\mathfrak{h}^{(c_1,\ldots,c_k)}_{(d_1,\ldots,d_k)}\colon \mathscr{I}_\lambda^n(S) \to \mathscr{I}_\lambda^n(S), \; \alpha\mapsto
  \left(
\begin{array}{ccc}
c_1 & \ldots & c_k \\
1_S & \ldots & 1_S \\
a_1 & \ldots & a_k \\
\end{array}
\right)
\cdot \alpha \cdot
  \left(
\begin{array}{ccc}
b_1 & \ldots & b_k \\
1_S & \ldots & 1_S \\
d_1 & \ldots & d_k \\
\end{array}
\right)
\end{equation*}
and
\begin{equation*}
  ^{(c_1,\ldots,c_k)}_{(d_1,\ldots,d_k)}\mathfrak{h}^{(a_1,\ldots,a_k)}_{(b_1,\ldots,b_k)}\colon \mathscr{I}_\lambda^n(S) \to \mathscr{I}_\lambda^n(S), \; \alpha\mapsto
  \left(
\begin{array}{ccc}
a_1 & \ldots & a_k \\
1_S & \ldots & 1_S \\
c_1 & \ldots & c_k \\
\end{array}
\right)
\cdot \alpha \cdot
  \left(
\begin{array}{ccc}
d_1 & \ldots & d_k \\
1_S & \ldots & 1_S \\
b_1 & \ldots & b_k \\
\end{array}
\right)
\end{equation*}
on the subspaces $S^{(a_1,\ldots,a_k)}_{(b_1,\ldots,b_k)}$ and $S^{(c_1,\ldots,c_k)}_{(d_1,\ldots,d_k)}$, respectively, are mutually inverse, and hence $S^{(a_1,\ldots,a_k)}_{(b_1,\ldots,b_k)}$ and $S^{(c_1,\ldots,c_k)}_{(d_1,\ldots,d_k)}$ are homeomorphic subspaces in $\mathscr{I}_\lambda^n(S)$. Also, it is obvious that in the case of subsemigroups $S^{(a_1,\ldots,a_k)}_{(a_1,\ldots,a_k)}$ and $S^{(c_1,\ldots,c_k)}_{(c_1,\ldots,c_k)}$ so defined restrictions of maps are topological isomorphisms.
\end{proof}

For any ordered collections of $k$ distinct elements $(a_1,\ldots,a_k)$ and $(b_1,\ldots,b_k)$ of $\lambda^k$ we define a map
\begin{equation*}
  \mathfrak{f}^{(a_1,\ldots,a_k)}_{(b_1,\ldots,b_k)}\colon \mathscr{I}_\lambda^n(S) \to \mathscr{I}_\lambda^n(S), \; \alpha\mapsto
  \left(
\begin{array}{ccc}
a_1 & \ldots & a_k \\
1_S & \ldots & 1_S \\
a_1 & \ldots & a_k \\
\end{array}
\right)
\cdot \alpha \cdot
  \left(
\begin{array}{ccc}
b_1 & \ldots & b_k \\
1_S & \ldots & 1_S \\
b_1 & \ldots & b_k \\
\end{array}
\right).
\end{equation*}

Proposition~\ref{proposition-5.1} implies the following corollary.

\begin{corollary}\label{corollary-5.4}
Let $S$ be a monoid, $\lambda$ be any non-zero cardinal, $n$ be an arbitrary positive integer $\leqslant\lambda$, $0<k\leqslant n$ and $\mathscr{I}_\lambda^n(S)$ be a Hausdorff semitopological semigroup. Then the set
\begin{equation*}
  {\Uparrow}S^{(a_1,\ldots,a_k)}_{(b_1,\ldots,b_k)}= \left(S^{(a_1,\ldots,a_k)}_{(b_1,\ldots,b_k)}\right)\left(\mathfrak{f}^{(a_1,\ldots,a_k)}_{(b_1,\ldots,b_k)}\right)^{-1}
\end{equation*}
is open in $\mathscr{I}_\lambda^n(S)$
for any ordered collections of $k$ distinct elements $(a_1,\ldots,a_k)$ and $(b_1,\ldots,b_k)$ of $\lambda^k$.
\end{corollary}

We recall that a topological space $X$ is said to be
\begin{itemize}
  \item \emph{compact} if each open cover of $X$ has a finite subcover;
  \item \emph{H-closed} if $X$ is a closed subspace of every Hausdorff topological space in which it contained.
\end{itemize}
It is well known that every Hausdorff compact space is H-closed, and every regular H-closed topological space is compact (see \cite[3.12.5]{Engelking-1989}).

\begin{lemma}\label{lemma-5.5}
Let $S$ be a monoid, $\lambda$ be any non-zero cardinal, $n$ be an arbitrary positive integer $\leqslant\lambda$, $0<k\leqslant n$ and $\mathscr{I}_\lambda^n(S)$ be a Hausdorff semitopological semigroup. If $S^{(a)}_{(b)}$ is a closed subset of $\mathscr{I}_\lambda^n(S)$ for any $a,b\in\lambda$ then $S^{(a_1,\ldots,a_k)}_{(b_1,\ldots,b_k)}$ is a closed subspace of $\mathscr{I}_\lambda^n(S)$ for any ordered collections of $k$ distinct elements $(a_1,\ldots,a_k)$ and $(b_1,\ldots,b_k)$ of $\lambda^k$.
\end{lemma}

\begin{proof}
For any $a,b\in\lambda$ the map
\begin{equation*}
  \mathfrak{f}^{(a)}_{(b)}\colon \mathscr{I}_\lambda^n(S) \to \mathscr{I}_\lambda^n(S), \; \alpha\mapsto
  \left(
\begin{array}{c}
a   \\
1_S \\
a \\
\end{array}
\right)
\cdot \alpha \cdot
  \left(
\begin{array}{ccc}
b  \\
1_S\\
b\\
\end{array}
\right)
\end{equation*}
is continuous, because $\mathscr{I}_\lambda^n(S)$ is a semitopological semigroup. This and  Proposition~\ref{proposition-5.1} imply that
\begin{equation*}
  S^{(a_1,\ldots,a_k)}_{(b_1,\ldots,b_k)}=\Big(S^{(a_1)}_{(b_1)}\Big)\Big(\mathfrak{f}^{(a_1)}_{(b_1)}\Big)^{-1}\cap\cdots\cap \Big(S^{(a_k)}_{(b_k)}\Big)\Big(\mathfrak{f}^{(a_k)}_{(b_k)}\Big)^{-1} \cap \mathscr{I}_\lambda^k(S)
\end{equation*}
a closed subspace of $\mathscr{I}_\lambda^n(S)$.
\end{proof}

Since a continuous image of a compact (an H-closed) space is compact (H-closed) (see \cite[Chapter~3]{Engelking-1989}), Proposition~\ref{proposition-5.3} and Lemma~\ref{lemma-5.5} imply the following corollary.

\begin{corollary}\label{corollary-5.6}
Let $S$ be a monoid, $\lambda$ be any non-zero cardinal, $n$ be an arbitrary positive integer $\leqslant\lambda$, $0<k\leqslant n$ and $\mathscr{I}_\lambda^n(S)$ be a Hausdorff semitopological semigroup. If the set $S^{(a)}_{(b)}$ is H-closed (compact) in $\mathscr{I}_\lambda^n(S)$ for some $a,b\in\lambda$ then $S^{(a_1,\ldots,a_k)}_{(b_1,\ldots,b_k)}$ is a closed subspace of $\mathscr{I}_\lambda^n(S)$ for any ordered collections of $k$ distinct elements $(a_1,\ldots,a_k)$ and $(b_1,\ldots,b_k)$ of $\lambda^k$.
\end{corollary}

\begin{definition}\label{definition-5.7}
Let $\mathfrak{S}$ be a class of semitopological semigroups.
Let $\lambda\geqslant 1$ be a cardinal,  $n$ be a positive integer $\leqslant \lambda$, and $(S,\tau)\in\mathfrak{S}$. Let $\tau_{\mathscr{I}}$ be a topology on $\mathscr{I}_\lambda^n(S)$ such that
\begin{itemize}
  \item[a)] $\left(\mathscr{I}_\lambda^n(S), \tau_{\mathscr{I}}\right)\in\mathfrak{S}$;
  \item[b)] the topological subspace $\left(S_{(a)}^{(a)},\tau_{B}|_{S_{\alpha,\alpha}}\right)$ is naturally homeomorphic to $(S,\tau)$ for some $a\in{\lambda}$, i.e., the map $\mathfrak{H}\colon S\to \mathscr{I}_\lambda^n(S)$,
      $s\mapsto
      \left(
\begin{array}{c}
a   \\
s \\
a \\
\end{array}
\right)
      $ is a topological embedding.
\end{itemize}
Then $\left(\mathscr{I}_\lambda^n(S), \tau_{\mathscr{I}}\right)$ is called a {\it topological $\mathscr{I}_\lambda^n$-extension of $(S, \tau)$ in $\mathfrak{S}$}.
\end{definition}

\begin{lemma}\label{lemma-5.8}
Let $(S,\tau)$ be a semitopological monoid, $\lambda$ be any non-zero cardinal, $n$ be an arbitrary positive integer $\leqslant\lambda$, $0<k\leqslant n$ and $\left(\mathscr{I}_\lambda^n(S), \tau_{\mathscr{I}}\right)$ be a topological $\mathscr{I}_\lambda^n$-extension of $(S, \tau)$ in the class of semitopological semigroups. Let $U_1(s_1),\ldots,U_k(s_k)$ be open neighbourhoods of the points $s_1,\ldots,s_k$ in $(S,\tau)$, respectively. Then the following sets
\begin{equation*}
  {\Uparrow}\left[U_1(s_1)\right]^{(a_1)}_{(b_1)}=\Big(\left[U_1(s_1)\right]^{(a_1)}_{(b_1)}\Big)\Big(\mathfrak{f}^{(a_1)}_{(b_1)}\Big)^{-1}, \qquad \ldots \quad, \qquad {\Uparrow}\left[U_k(s_k)\right]^{(a_k)}_{(b_k)}=\Big(\left[U_k(s_k)\right]^{(a_k)}_{(b_k)}\Big)\Big(\mathfrak{f}^{(a_k)}_{(b_k)}\Big)^{-1},
\end{equation*}
and
\begin{equation*}
  {\Uparrow}\left[U_1(s_1),\ldots,U_k(s_k)\right]^{(a_1,\ldots,a_k)}_{(b_1,\ldots,b_k)}={\Uparrow}\left[U_1(s_1)\right]^{(a_1)}_{(b_1)}
  \cap\ldots\cap{\Uparrow}\left[U_k(s_k)\right]^{(a_k)}_{(b_k)},
\end{equation*}
are open neighbourhoods of the points
\begin{equation*}
\left(
\begin{array}{c}
a_1 \\
s_1 \\
b_1 \\
\end{array}
\right), \cdots,
\left(
\begin{array}{c}
a_k \\
s_k \\
b_k \\
\end{array}
\right),
\qquad \hbox{and} \qquad
  \left(
\begin{array}{ccc}
a_1 & \ldots & a_k \\
s_1 & \ldots & s_k \\
b_1 & \ldots & b_k \\
\end{array}
\right)
\end{equation*}
in $\left(\mathscr{I}_\lambda^n(S), \tau_{\mathscr{I}}\right)$, respectively, for any ordered collections of $k$ distinct elements $(a_1,\ldots,a_k)$ and $(b_1,\ldots,b_k)$ of $\lambda^k$.
\end{lemma}

\begin{proof}
Since $\left(\mathscr{I}_\lambda^n(S), \tau_{\mathscr{I}}\right)$ be a topological $\mathscr{I}_\lambda^n$-extension of $(S, \tau)$ in the class of Hausdorff semitopological semigroups, there exist open neighbourhoods $W_1,\ldots,W_k$ of of the points
$
\left(
\begin{array}{c}
a_1 \\
s_1 \\
b_1 \\
\end{array}
\right), \cdots,
\left(
\begin{array}{c}
a_k \\
s_k \\
b_k \\
\end{array}
\right)
$
in $\left(\mathscr{I}_\lambda^n(S), \tau_{\mathscr{I}}\right)$, respectively, such that
\begin{equation*}
W_1\cap S^{(a_1)}_{(b_1)}=\left[U_1(s_1)\right]^{(a_1)}_{(b_1)}, \qquad \ldots\;, \qquad W_k\cap S^{(a_k)}_{(b_k)}=\left[U_k(s_k)\right]^{(a_k)}_{(b_k)}.
\end{equation*}
Then the requested statement of the lemma follows from the separate continuity of the semigroup operation in $\left(\mathscr{I}_\lambda^n(S), \tau_{\mathscr{I}}\right)$.
\end{proof}

\begin{theorem}\label{theorem-5.9}
Let $(S,\tau)$ be a Hausdorff compact semitopological monoid, $\lambda$ be any non-zero cardinal, $n$ be an arbitrary positive integer $\leqslant\lambda$, $0<k\leqslant n$ and $\left(\mathscr{I}_\lambda^n(S), \tau_{\mathscr{I}}\right)$ be a compact topological $\mathscr{I}_\lambda^n$-extension of $(S, \tau)$ in the class of Hausdorff semitopological semigroups. Then the subspace $S^{(a_1,\ldots,a_k)}_{(b_1,\ldots,b_k)}$ of $\left(\mathscr{I}_\lambda^n(S), \tau_{\mathscr{I}}\right)$ is compact and moreover it is homeomorphic to the power $S^k$ with the product topology by the mapping
\begin{equation*}
  \mathfrak{H}\colon S^{(a_1,\ldots,a_k)}_{(b_1,\ldots,b_k)}\to S^k, \;
  \left(
\begin{array}{ccc}
a_1 & \ldots & a_k \\
s_1 & \ldots & s_k \\
b_1 & \ldots & b_k \\
\end{array}
\right)\mapsto (s_1,\ldots,s_k),
\end{equation*}
for any ordered collections of $k$ distinct elements $(a_1,\ldots,a_k)$ and $(b_1,\ldots,b_k)$ of $\lambda^k$.
\end{theorem}

\begin{proof}
Since the monoid $(S,\tau)$ is compact Corollary~\ref{corollary-5.6} implies that $S^{(a_1,\ldots,a_k)}_{(b_1,\ldots,b_k)}$ a closed subset of of $\left(\mathscr{I}_\lambda^n(S), \tau_{\mathscr{I}}\right)$. Then compactness of of $\left(\mathscr{I}_\lambda^n(S), \tau_{\mathscr{I}}\right)$ implies that $S^{(a_1,\ldots,a_k)}_{(b_1,\ldots,b_k)}$ is compact, as well.

It is obvious that the above defined map $\mathfrak{H}\colon S^{(a_1,\ldots,a_k)}_{(b_1,\ldots,b_k)}\to S^k$ is a bijection. Also, Lemma~\ref{lemma-5.8} implies that the map $\mathfrak{H}$ is continuous, and it is a homeomorphism, because $S^k$ and $S^{(a_1,\ldots,a_k)}_{(b_1,\ldots,b_k)}$ are compacta.
\end{proof}

Proposition~\ref{proposition-5.1} and Theorem~\ref{theorem-5.9} imply the following corollary.

\begin{corollary}\label{corollary-5.10}
Let $(S,\tau)$ be a Hausdorff compact semitopological monoid, $\lambda$ be any non-zero cardinal, $n$ be an arbitrary positive integer $\leqslant\lambda$, $0<k\leqslant n$ and $\left(\mathscr{I}_\lambda^n(S), \tau_{\mathscr{I}}\right)$ be a compact topological $\mathscr{I}_\lambda^n$-extension of $(S, \tau)$ in the class of Hausdorff semitopological semigroups. Then $S^{(a_1,\ldots,a_k)}_{(b_1,\ldots,b_k)}$ is an open-and-closed subset of $\left(\mathscr{I}_\lambda^n(S), \tau_{\mathscr{I}}\right)$ for any ordered collections of $k$ distinct elements $(a_1,\ldots,a_k)$ and $(b_1,\ldots,b_k)$ of $\lambda^k$, and the space $\left(\mathscr{I}_\lambda^n(S), \tau_{\mathscr{I}}\right)$ is the topological sum of such sets with isolated zero.
\end{corollary}

\begin{remark}\label{remark-5.11}
Since by Theorem~ of \cite{Gutik-Pavlyk-2005} an infinite semigroup of matrix units and hence an infinite semigroup $\mathscr{I}_\lambda^n$ do not embed into  compact Hausdorff topological semigroups, Corollary~\ref{corollary-5.10} describes   compact topological $\mathscr{I}_\lambda^n$-extensions of compact semigroups $(S, \tau)$ in the class of Hausdorff topological semigroups.
\end{remark}

\begin{example}\label{example-5.12}
Let $(S,\tau_S)$ be a compact Hausdorff semitopological monoid. On the semigroup $\mathscr{I}_\lambda^n(S)$ we define a topology $\tau_{\mathscr{I}}^\mathbf{c}$ in the following way. Put
  \begin{equation*}
  \mathscr{P}_k^\mathbf{c}(0)=\left\{\mathscr{I}_\lambda^n(S)\setminus {\Uparrow}S^{(a_1,\ldots,a_k)}_{(b_1,\ldots,b_k)}\colon (a_1,\ldots,a_k) \hbox{~and~} (b_1,\ldots,b_k) \hbox{~are ordered collections of $k$ distinct elements of $\lambda^k$}\right\},
  \end{equation*}
  for any $k=1,\ldots,n$, and
  \begin{equation*}
  \mathscr{P}^\mathbf{c}(a,s,b)=\left\{{\Uparrow}\left[U(s)\right]^{(a)}_{(b)}\colon U(s) \hbox{~is an open neighbourhood of~} s \hbox{~in~} (S,\tau_S) \right\}, \quad \hbox{for some~}
    \left(
    \begin{array}{c}
      a \\
      s \\
      b \\
    \end{array}
  \right)\in \mathscr{I}_\lambda^n(S)\setminus\{0\}.
  \end{equation*}
The topology $\tau_{\mathscr{I}}^\mathbf{c}$ on $\mathscr{I}_\lambda^n(S)$ is generated by the family
\begin{equation*}
  \mathscr{P}^\mathbf{c}=\left\{\mathscr{P}_k^\mathbf{c}(0)\colon k=1,\ldots,n\right\}\cup \left\{\mathscr{P}^\mathbf{c}(a,s,b)\colon
  \left(
    \begin{array}{c}
      a \\
      s \\
      b \\
    \end{array}
  \right)\in \mathscr{I}_\lambda^n(S)\setminus\{0\}
  \right\},
  \end{equation*}
as a subbase.
\end{example}

\begin{remark}\label{remark-5.13}
Lemma~\ref{lemma-5.8} and the definition of the topology $\tau_{\mathscr{I}}^\mathbf{c}$ on $\mathscr{I}_\lambda^n(S)$ implies that the following statements hold.
\begin{itemize}
  \item[(1)] For any $k=1,\ldots,n$  and every ordered collection $(a_1,\ldots,a_k)$ and $(b_1,\ldots,b_k)$ of $k$ distinct elements of $\lambda^k$ the set ${\Uparrow}S^{(a_1,\ldots,a_k)}_{(b_1,\ldots,b_k)}$ is closed in $\left(\mathscr{I}_\lambda^n(S),\tau_{\mathscr{I}}^\mathbf{c}\right)$.
  \item[(2)] For any element
  $\alpha_S=
  \left(
\begin{array}{ccc}
a_1 & \ldots & a_k \\
s_1 & \ldots & s_k \\
b_1 & \ldots & b_k \\
\end{array}
\right)
  $ of $\mathscr{I}_\lambda^n(S)$ and any open neighbourhoods $U_1(s_1),\ldots,U_k(s_k)$ of the points $s_1,\ldots,s_k$ in $(S,\tau)$ the set ${\Uparrow}\left[U_1(s_1),\ldots,U_k(s_k)\right]^{(a_1,\ldots,a_k)}_{(b_1,\ldots,b_k)}\setminus \left({\Uparrow}S^{(a^1_1,\ldots,a^1_{l_1})}_{(b^1_1,\ldots,b^1_{l_1})}\cup\cdots\cup {\Uparrow}S^{(a^p_1,\ldots,a^p_{l_p})}_{(b^p_1,\ldots,b^p_{l_p})}\right)$ such that $\alpha_S\notin {\Uparrow}S^{(a^1_1,\ldots,a^1_{l_1})}_{(b^1_1,\ldots,b^1_{l_1})}\cup\cdots\cup {\Uparrow}S^{(a^p_1,\ldots,a^p_{l_p})}_{(b^p_1,\ldots,b^p_{l_p})}$, is an open neighbourhood of the point  $\alpha_S$ in $\left(\mathscr{I}_\lambda^n(S),\tau_{\mathscr{I}}^\mathbf{c}\right)$. Here we have that $\{a_1,\ldots,a_k\}\subsetneqq \left\{a^j_1,\ldots,a^j_{l_j}\right\}$ and $\{b_1,\ldots,b_k\}\subsetneqq \left\{b^j_1,\ldots,b^j_{l_j}\right\}$ for all $j=1,\ldots,p$.
\end{itemize}
\end{remark}

\begin{theorem}\label{theorem-5.14}
If $(S,\tau_S)$ is a compact Hausdorff semitopological monoid then $\left(\mathscr{I}_\lambda^n(S),\tau_{\mathscr{I}}^\mathbf{c}\right)$ is a compact Hausdorff semitopological semigroup.
\end{theorem}

\begin{proof}
It is obvious that the topology $\tau_{\mathscr{I}}^\mathbf{c}$ is Hausdorff.

By the Alexander Subbase Theorem (see \cite[3.12.2]{Engelking-1989}) it is sufficient to show that every open cover  of $\mathscr{I}_\lambda^n(S)$ which consists of elements of the subbase $\mathscr{P}^\mathbf{c}$ has a finite subcover.

We shall show that the space $\left(\mathscr{I}_\lambda^n(S),\tau_{\mathscr{I}}^\mathbf{c}\right)$ is compact by induction. In the case when $n=1$, Corollary~13 from \cite{Gutik-Pavlyk-2013a} implies that the space $\left(\mathscr{I}_\lambda^1(S),\tau_{\mathscr{I}}^\mathbf{c}\right)$ is compact. Next we shall show the step of induction: $\left(\mathscr{I}_\lambda^{k-1}(S),\tau_{\mathscr{I}}^\mathbf{c}\right)$ is compact implies $\left(\mathscr{I}_\lambda^k(S),\tau_{\mathscr{I}}^\mathbf{c}\right)$ is compact, too, for $k=2,\ldots,n$. Without loss of generality we my assume that $k=n$.

Let $\mathscr{U}$ be an arbitrary open cover of $\left(\mathscr{I}_\lambda^n(S),\tau_{\mathscr{I}}^\mathbf{c}\right)$ which consists of elements of the subbase $\mathscr{P}^\mathbf{c}$. The assumption of induction implies that there exists a finite subfamily $\mathscr{U}_{n-1}$ of $\mathscr{U}$ which is a subcover of $\mathscr{I}_\lambda^{n-1}(S)$. Fix an arbitrary element $V_0=\mathscr{I}_\lambda^n(S)\setminus {\Uparrow}S^{(a_1,\ldots,a_p)}_{(b_1,\ldots,b_p)}\in\mathscr{U}_{n-1}$ which contains the zero $0$ of $\mathscr{I}_\lambda^n(S)$. Then $p\in\{1,\ldots,n\}$.

We observe that an arbitrary element $U_0$ of  the family $\left\{\mathscr{P}_k^\mathbf{c}(0)\colon k=1,\ldots,n\right\}$ contains the set $S^{(a_1,\ldots,a_p)}_{(b_1,\ldots,b_p)}$ if and only if $U_0\cap S^{(a_1,\ldots,a_p)}_{(b_1,\ldots,b_p)}\neq\varnothing$. This implies that only one of the following conditions holds:
\begin{itemize}
  \item[(1)] there does not exist an element of $\mathscr{U}_{n-1}$ from the family $\left\{\mathscr{P}_k^\mathbf{c}(0)\colon k=1,\ldots,n\right\}$ which contains the set $S^{(a_1,\ldots,a_p)}_{(b_1,\ldots,b_p)}$;
  \item[(2)]  there exists $W_0\in \mathscr{U}_{n-1}\cap\left\{\mathscr{P}_k^\mathbf{c}(0)\colon k=1,\ldots,n\right\}$ such that  $S^{(a_1,\ldots,a_p)}_{(b_1,\ldots,b_p)}\subseteq W_0$.
\end{itemize}

Suppose that condition $(1)$ holds. First we consider the case when $p<n$. By Theorem~\ref{theorem-5.9}, $S^{(a_1,\ldots,a_p)}_{(b_1,\ldots,b_p)}$ is compact, and hence there exists finitely many elements ${\Uparrow}\left[U(s_1)\right]^{(c_1)}_{(d_1)},\ldots,{\Uparrow}\left[U(s_m)\right]^{(c_m)}_{(d_m)}$ in $\mathscr{U}_{n-1}\cap \mathscr{P}^\mathbf{c}\setminus\left\{\mathscr{P}_k^\mathbf{c}(0)\colon k=1,\ldots,n\right\}$ such that $S^{(a_1,\ldots,a_p)}_{(b_1,\ldots,b_p)}\subseteq {\Uparrow}\left[U(s_1)\right]^{(c_1)}_{(d_1)}\cup\cdots\cup{\Uparrow}\left[U(s_m)\right]^{(c_m)}_{(d_m)}$. It is obvious that $\left\{U_0,{\Uparrow}\left[U(s_1)\right]^{(c_1)}_{(d_1)},\ldots,{\Uparrow}\left[U(s_m)\right]^{(c_m)}_{(d_m)}\right\}$ is a finite cover of $\left(\mathscr{I}_\lambda^n(S),\tau_{\mathscr{I}}^\mathbf{c}\right)$.

Next, we consider case $p=n$. We identify the set $S^{(a_1,\ldots,a_n)}_{(b_1,\ldots,b_n)}$ and the power $S^n$ by the mapping
\begin{equation}\label{eq-5.1}
  \mathfrak{H}\colon S^{(a_1,\ldots,a_n)}_{(b_1,\ldots,b_n)}\to S^n, \;
  \left(
\begin{array}{ccc}
a_1 & \ldots & a_n \\
s_1 & \ldots & s_n \\
b_1 & \ldots & b_n \\
\end{array}
\right)\mapsto (s_1,\ldots,s_n).
\end{equation}
The semigroup operation of $\mathscr{I}_\lambda^n(S)$ implies that ${\Uparrow}\left[U(s)\right]^{(c)}_{(d)}\cap S^{(a_1,\ldots,a_n)}_{(b_1,\ldots,b_n)}\neq\varnothing$ if and only if $c=a_i$ and $d=b_i$ for some $i=1,\ldots,n$. Then by \eqref{eq-5.1} for every $i=1,\ldots,n$ we have that
\begin{equation}\label{eq-5.2}
  \left({\Uparrow}\left[U(s)\right]^{(a_i)}_{(b_i)}\cap S^{(a_1,\ldots,a_n)}_{(b_1,\ldots,b_n)}\right)\mathfrak{H}=S\times\cdots\times \underbrace{U(s)}_{i-\hbox{th}}\times\cdots\times S\subseteq S^n.
\end{equation}
Then the subbase $\mathscr{P}^\mathbf{c}$ on $\mathscr{I}_\lambda^n(S)$ and the map \eqref{eq-5.1} determine the product topology on $S^n$ from the space $S$, and hence the space $S^n$ is compact.

Suppose that $S^{(a_1,\ldots,a_n)}_{(b_1,\ldots,b_n)}$ is not compact. Then $S^{(a_1,\ldots,a_n)}_{(b_1,\ldots,b_n)}$ has a cover $\mathscr{W}$ which consists of the open sets of the form ${\Uparrow}\left[U(s)\right]^{(c)}_{(d)}$ and $\mathscr{W}$ does not have a finite subcover. Then the cover $\mathscr{W}_{S^n}$ of $S^n$ which determines by formula \eqref{eq-5.2} from the family $\mathscr{W}$ has no finite subcover, too. This contradicts the compactness of $S^n$.

Hence in case $(1)$ the cover $\mathscr{U}$ of $\mathscr{I}_\lambda^n(S)$ has a finite subcover.

Suppose that condition $(2)$ holds. Then $W_0=\mathscr{I}_\lambda^n(S)\setminus {\Uparrow}S^{(c_1,\ldots,c_q)}_{(d_1,\ldots,d_q)}$ with $q\leqslant n$. If ${\Uparrow}S^{(c_1,\ldots,c_q)}_{(d_1,\ldots,d_q)}\cap {\Uparrow}S^{(a_1,\ldots,a_p)}_{(b_1,\ldots,b_p)}=\varnothing$ then $\{V_0,W_0\}$ is a cover of $\mathscr{I}_\lambda^n(S)$. In the other case there exists a smallest positive integer $p_1$ such that $\max\{p+1,q\}\leqslant p_1\leqslant n$ and two ordered $p_1$-collections of distinct elements $(e_1,\ldots,e_{p_1})$ and $(f_1,\ldots,f_{p_1})$ of the power $\lambda^{p_1}$ such that
${\Uparrow}S^{(c_1,\ldots,c_q)}_{(d_1,\ldots,d_q)}\cap {\Uparrow}S^{(a_1,\ldots,a_p)}_{(b_1,\ldots,b_p)}= {\Uparrow}S^{(e_1,\ldots,e_{p_1})}_{(f_1,\ldots,f_{p_1})}$. Then for the open set $U_1=U_0\cup W_0=\mathscr{I}_\lambda^n(S)\setminus {\Uparrow}S^{(e_1,\ldots,e_{p_1})}_{(f_1,\ldots,f_{p_1})}$ either condition $(1)$ or condition $(2)$ holds.

Since $p+1\leqslant p_1\leqslant n$, we repeating finitely many items the above procedure we get that the space  $\left(\mathscr{I}_\lambda^n(S),\tau_{\mathscr{I}}^\mathbf{c}\right)$ is compact.

Next we shall show that the topology $\tau_{\mathscr{I}}^\mathbf{c}$ is shift-continuous on $\left(\mathscr{I}_\lambda^n(S),\tau_{\mathscr{I}}^\mathbf{c}\right)$. We consider all possible cases.

$(i)$ $0\cdot 0=0$. Then for any open neighbourhood $U_0$ of zero in $\left(\mathscr{I}_\lambda^n(S),\tau_{\mathscr{I}}^\mathbf{c}\right)$ we have that
\begin{equation*}
U_0\cdot 0=0\cdot U_0= \{0\}\subseteq U_0.
\end{equation*}

$(ii)$ $\alpha\cdot 0=0$. Then for any open neighbourhoods $U_0$ and $U_\alpha$ of zero and $\alpha$ in $\left(\mathscr{I}_\lambda^n(S),\tau_{\mathscr{I}}^\mathbf{c}\right)$, respectively, we have that
\begin{equation*}
U_\alpha\cdot 0=\{0\}\subseteq U_0.
\end{equation*}
Let $W_0=\mathscr{I}_\lambda^n(S)\setminus\left({\Uparrow}S^{(a^1_1,\ldots,a^1_{p_1})}_{(b^1_1,\ldots,b^1_{p_1})}\cup \cdots\cup {\Uparrow}S^{(a^k_1,\ldots,a^k_{p_k})}_{(b^k_1,\ldots,b^k_{p_k})}\right)$ be an arbitrary basic neighbourhood of $0$ in $\left(\mathscr{I}_\lambda^n(S),\tau_{\mathscr{I}}^\mathbf{c}\right)$. Without loss of generality we may assume that $p_1,\ldots,p_k\leqslant |\mathbf{d}(\alpha)|$. Put
\begin{equation*}
  \mathbf{B}=\left\{S^{(a)}_{(b)}\colon a\in \mathbf{d}(\alpha)\quad \hbox{and} \quad b\in \left\{b^1_1,\ldots,b^1_{p_1},\ldots,b^k_1,\ldots,b^k_{p_k}\right\}\right\}.
\end{equation*}
Then the family $\mathbf{B}$ is finite and $\alpha\cdot U_0\subseteq W_0$ for $U_0=\mathscr{I}_\lambda^n(S)\setminus\bigcup_{S^{(a)}_{(b)}\in \mathbf{B}}{\Uparrow}S^{(a)}_{(b)}$.

$(iii)$ $0\cdot\alpha=0$. Then for any open neighbourhoods $U_0$ and $U_\alpha$ of zero and $\alpha$ in $\left(\mathscr{I}_\lambda^n(S),\tau_{\mathscr{I}}^\mathbf{c}\right)$, respectively, we have that
\begin{equation*}
0\cdot U_\alpha=\{0\}\subseteq U_0.
\end{equation*}
Let $W_0=\mathscr{I}_\lambda^n(S)\setminus\left({\Uparrow}S^{(a^1_1,\ldots,a^1_{p_1})}_{(b^1_1,\ldots,b^1_{p_1})}\cup \cdots \cup {\Uparrow}S^{(a^k_1,\ldots,a^k_{p_k})}_{(b^k_1,\ldots,b^k_{p_k})}\right)$ be an arbitrary basic neighbourhood of $0$ in $\left(\mathscr{I}_\lambda^n(S),\tau_{\mathscr{I}}^\mathbf{c}\right)$. Without loss of generality we may assume that $p_1,\ldots,p_k\leqslant |\mathbf{d}(\alpha)|$. Put
\begin{equation*}
  \mathbf{B}=\left\{S^{(a)}_{(b)}\colon b\in \mathbf{r}(\alpha)\quad \hbox{and} \quad a\in \left\{a^1_1,\ldots,a^1_{p_1},\ldots,a^k_1,\ldots,a^k_{p_k}\right\}\right\}.
\end{equation*}
Then the family $\mathbf{B}$ is finite and $U_0\cdot\alpha\subseteq W_0$ for $U_0=\mathscr{I}_\lambda^n(S)\setminus\bigcup_{S^{(a)}_{(b)}\in \mathbf{B}}{\Uparrow}S^{(a)}_{(b)}$.

$(iv)$ $\alpha\cdot\beta=0$. Fix an arbitrary open neighbourhood $W_0$ of $0$ in $\left(\mathscr{I}_\lambda^n(S),\tau_{\mathscr{I}}^\mathbf{c}\right)$. Without loss of generality we may assume that $W_0=\mathscr{I}_\lambda^n(S)\setminus\left({\Uparrow}S^{(a_1)}_{(b_1)}\cup \cdots \cup {\Uparrow}S^{(a_k)}_{(b_k)}\right)$. Since $\alpha\cdot\beta=0$ we have that $\mathbf{r}(\alpha)\cap \mathbf{d}(\beta)=\varnothing$.  We put
\begin{equation*}
  \mathbf{B}_\alpha=\left\{S^{(a)}_{(b)}\colon a\in\{a_1,\ldots, a_k\}, b\in\mathbf{d}(\beta),  \hbox{~and~} \alpha\notin {\Uparrow}S^{(a)}_{(b)}\right\}
\end{equation*}
and
\begin{equation*}
  \mathbf{B}_\beta=\left\{S^{(a)}_{(b)} \colon b\in\{b_1,\ldots,b_k\}, a\in\mathbf{r}(\alpha), \hbox{~and~}  \beta\notin {\Uparrow}S^{(a)}_{(b)}\right\}.
\end{equation*}
Let $S^{(a_1,\ldots,a_k)}_{(b_1,\ldots,b_k)}$ and $S^{(c_1,\ldots,c_p)}_{(d_1,\ldots,d_p)}$, $1\leqslant k,p\leqslant n$, such that $\alpha\in S^{(a_1,\ldots,a_k)}_{(b_1,\ldots,b_k)}$ and $\beta\in S^{(c_1,\ldots,c_p)}_{(d_1,\ldots,d_p)}$.
Then the families $\mathbf{B}_\alpha$ and $\mathbf{B}_\beta$ are finite, and hence by Remark~\ref{remark-5.13}(2) the sets $V_\alpha=S^{(a_1,\ldots,a_k)}_{(b_1,\ldots,b_k)}\setminus\bigcup_{S^{(a)}_{(b)}\in \mathbf{B}_\alpha}{\Uparrow}S^{(a)}_{(b)}$ and $V_\beta=S^{(c_1,\ldots,c_p)}_{(d_1,\ldots,d_p)}\setminus\bigcup_{S^{(a)}_{(b)}\in \mathbf{B}_\beta}{\Uparrow}S^{(a)}_{(b)}$ are open neighbourhoods of the points $\alpha$ and $\beta$ in $\left(\mathscr{I}_\lambda^n(S),\tau_{\mathscr{I}}^\mathbf{c}\right)$, respectively, such that
\begin{equation*}
  V_\alpha\cdot \beta\subseteq W_0 \qquad \hbox{and} \qquad\alpha\cdot V_\beta\subseteq W_0.
\end{equation*}

$(v)$ $\alpha\cdot\beta=\gamma\neq0$ and $\mathbf{r}(\alpha)=\mathbf{d}(\beta)$. Without loss of generality we may assume that
$\alpha=
\left(
\begin{array}{ccc}
a_1 & \ldots & a_k \\
s_1 & \ldots & s_k \\
b_1 & \ldots & b_k \\
\end{array}
\right)
$
and
$\beta=
\left(
\begin{array}{ccc}
b_1 & \ldots & b_k \\
t_1 & \ldots & t_k \\
c_1 & \ldots & c_k \\
\end{array}
\right)$, and hence we have that
$\gamma=
\left(
\begin{array}{ccc}
a_1    & \ldots & a_k    \\
s_1t_1 & \ldots & s_kt_k \\
c_1    & \ldots & c_k    \\
\end{array}
\right)
$. Then for any open neighbourhood $U_\gamma={\Uparrow}\left[U_1(s_1t_1),\ldots,U_k(s_kt_k)\right]^{(a_1,\ldots,a_k)}_{(c_1,\ldots,c_k)}\setminus \left({\Uparrow}S^{(a^1_1,\ldots,a^1_{l_1})}_{(b^1_1,\ldots,b^1_{l_1})}\cup\cdots\cup {\Uparrow}S^{(a^p_1,\ldots,a^p_{l_p})}_{(b^p_1,\ldots,b^p_{l_p})}\right)$ of $\gamma$ in $\left(\mathscr{I}_\lambda^n(S),\tau_{\mathscr{I}}^\mathbf{c}\right)$ we have that
\begin{equation*}
  {\Uparrow}\left[V_1(s_1),\ldots,V_k(s_k)\right]^{(a_1,\ldots,a_k)}_{(b_1,\ldots,b_k)}\cdot \beta \subseteq {\Uparrow}\left[U_1(s_1t_1),\ldots,U_k(s_kt_k)\right]^{(a_1,\ldots,a_k)}_{(c_1,\ldots,c_k)}\cap S^{(a_1,\ldots,a_k)}_{(c_1,\ldots,c_k)} \subseteq U_\gamma
\end{equation*}
and
\begin{equation*}
  \alpha\cdot {\Uparrow}\left[V_1(t_1),\ldots,V_k(t_k)\right]^{(b_1,\ldots,b_k)}_{(c_1,\ldots,c_k)}\subseteq {\Uparrow}\left[U_1(s_1t_1),\ldots,U_k(s_kt_k)\right]^{(a_1,\ldots,a_k)}_{(c_1,\ldots,c_k)}\cap S^{(a_1,\ldots,a_k)}_{(c_1,\ldots,c_k)} \subseteq U_\gamma,
\end{equation*}
where $V_1(s_1),\ldots,V_k(s_k), V_1(t_1),\ldots,V_k(t_k)$ are open neighbourhoods of the points $s_1, \ldots, s_k, t_1, \ldots, t_k$ in $(S,\tau_S)$, respectively, such that
\begin{equation*}
  V_1(s_1)\cdot t_1\subseteq U_1(s_1t_1),\ldots,V_k(s_k)\cdot t_k\subseteq U_k(s_kt_k) \qquad \hbox{and} \qquad
  s_1\cdot V_1(t_1)\subseteq U_1(s_1t_1),\ldots,s_k\cdot V_k(t_k)\subseteq U_k(s_kt_k).
\end{equation*}

$(vi)$ $\alpha\cdot\beta=\gamma\neq0$ and $\mathbf{r}(\alpha)\subsetneqq\mathbf{d}(\beta)$. Without loss of generality we may assume that
$\alpha=
\left(
\begin{array}{ccc}
a_1 & \ldots & a_k \\
s_1 & \ldots & s_k \\
b_1 & \ldots & b_k \\
\end{array}
\right)
$
and
$\beta=
\left(
\begin{array}{cccccc}
b_1 & \ldots & b_k & b_{k+1} & \ldots & b_{k+j}\\
t_1 & \ldots & t_k & t_{k+1} & \ldots & t_{k+j}\\
c_1 & \ldots & c_k & c_{k+1} & \ldots & c_{k+j}\\
\end{array}
\right)$, where $1\leqslant j\leqslant n-k,$ and hence we have that
$\gamma=
\left(
\begin{array}{ccc}
a_1    & \ldots & a_k    \\
s_1t_1 & \ldots & s_kt_k \\
c_1    & \ldots & c_k    \\
\end{array}
\right)
$. Then for any open neighbourhood $U_\gamma={\Uparrow}\left[U_1(s_1t_1),\ldots,U_k(s_kt_k)\right]^{(a_1,\ldots,a_k)}_{(c_1,\ldots,c_k)}\setminus \left({\Uparrow}S^{(a^1_1,\ldots,a^1_{l_1})}_{(b^1_1,\ldots,b^1_{l_1})}\cup\cdots\cup {\Uparrow}S^{(a^p_1,\ldots,a^p_{l_p})}_{(b^p_1,\ldots,b^p_{l_p})}\right)$ of the point $\gamma$ in $\left(\mathscr{I}_\lambda^n(S),\tau_{\mathscr{I}}^\mathbf{c}\right)$ we have that
\begin{equation*}
  \alpha\cdot {\Uparrow}\left[V_1(t_1),\ldots,V_k(t_k)\right]^{(b_1,\ldots,b_k)}_{(c_1,\ldots,c_k)}\subseteq {\Uparrow}\left[U_1(s_1t_1),\ldots,U_k(s_kt_k)\right]^{(a_1,\ldots,a_k)}_{(c_1,\ldots,c_k)}\cap S^{(a_1,\ldots,a_k)}_{(c_1,\ldots,c_k)} \subseteq U_\gamma,
\end{equation*}
where $V_1(t_1),\ldots,V_k(t_k)$ are open neighbourhoods of the points $t_1, \ldots, t_k$ in $(S,\tau_S)$, respectively, such that
\begin{equation*}
  s_1\cdot V_1(t_1)\subseteq U_1(s_1t_1),\ldots,s_k\cdot V_k(t_k)\subseteq U_k(s_kt_k).
\end{equation*}

Fix an arbitrary open neighbourhood $U_\gamma$ of the point $\gamma$ in $\left(\mathscr{I}_\lambda^n(S),\tau_{\mathscr{I}}^\mathbf{c}\right)$. Then Lemma~\ref{lemma-5.8} implies that without loss of generality we may assume that
\begin{equation*}
 U_\gamma={\Uparrow}\left[U_1(s_1t_1),\ldots,U_k(s_kt_k)\right]^{(a_1,\ldots,a_k)}_{(c_1,\ldots,c_k)}\setminus \left({\Uparrow}S^{(a_1,\ldots,a_k,x_1)}_{(c_1,\ldots,c_k,y_1)}\cup\cdots\cup {\Uparrow}S^{(a_1,\ldots,a_k,x_p)}_{(c_1,\ldots,c_k,y_p)}\right)
\end{equation*}
for some $x_1,\ldots,x_p\in\lambda\setminus \{a_1,\ldots,a_k\}$ and $y_1,\ldots,y_p\in\lambda\setminus \{c_1,\ldots,c_k\}$. We put
\begin{equation*}
  \mathbf{B}_\alpha=\left\{S^{(a_1,\ldots,a_k,a)}_{(b_1,\ldots,b_k,b)}\colon a\in\left\{x_1,\ldots,x_p\right\} \quad \hbox{and} \quad b\in\left\{b_{k+1},\ldots,b_{k+j}\right\}\right\}.
\end{equation*}
It is obvious that the family $\mathbf{B}_\alpha$ is finite. Then $V_\alpha\cdot \beta\subseteq U_\gamma$ for
\begin{equation*}
  V_\alpha={\Uparrow}\left[V_1(s_1),\ldots,V_k(s_k)\right]^{(a_1,\ldots,a_k)}_{(b_1,\ldots,b_k)}\setminus \bigcup_{S^{(a_1,\ldots,a_k,a)}_{(b_1,\ldots,b_k,b)}\in \mathbf{B}_\alpha}{{\Uparrow}}S^{(a_1,\ldots,a_k,a)}_{(b_1,\ldots,b_k,b)},
\end{equation*}
where $V_1(s_1),\ldots,V_k(s_k)$ are open neighbourhoods of the points $s_1, \ldots, s_k$ in $(S,\tau_S)$, respectively, such that
\begin{equation*}
  V_1(s_1)\cdot t_1\subseteq U_1(s_1t_1),\ldots,V_k(s_k)\cdot t_k\subseteq U_k(s_kt_k).
\end{equation*}

$(vii)$ $\alpha\cdot\beta=\gamma\neq0$ and $\mathbf{d}(\beta)\subsetneqq\mathbf{r}(\alpha)$. In this case the proof of separate continuity of the semigroup operation is similar to case $(vi)$.

$(viii)$ $\alpha\cdot\beta=\gamma\neq0$, $\mathbf{d}(\gamma)\subsetneqq\mathbf{d}(\alpha)$ and $\mathbf{r}(\gamma)\subsetneqq\mathbf{r}(\beta)$. Without loss of generality we may assume that
$\alpha=
\left(
\begin{array}{cccccc}
a_1 & \ldots & a_k & a_{k+11} & \ldots & a_{k+m}\\
s_1 & \ldots & s_k & s_{k+11} & \ldots & s_{k+m}\\
b_1 & \ldots & b_k & b_{k+11} & \ldots & b_{k+m}\\
\end{array}
\right)
$,
$\beta=
\left(
\begin{array}{cccccc}
b_1 & \ldots & b_k & b_{k+1} & \ldots & b_{k+j}\\
t_1 & \ldots & t_k & t_{k+1} & \ldots & t_{k+j}\\
c_1 & \ldots & c_k & c_{k+1} & \ldots & c_{k+j}\\
\end{array}
\right)$ and
$\gamma=
\left(
\begin{array}{ccc}
a_1    & \ldots & a_k    \\
s_1t_1 & \ldots & s_kt_k \\
c_1    & \ldots & c_k    \\
\end{array}
\right)
$, where $1\leqslant j,m\leqslant n-k$. We put
$\varepsilon=
\left(
\begin{array}{ccc}
b_1 & \ldots & b_k \\
1_S & \ldots & 1_S \\
b_1 & \ldots & b_k \\
\end{array}
\right)
$, where $1_S$ is the unit element of $S$. It is obvious that $\gamma=\alpha\cdot\varepsilon\cdot\beta$. Hence, in this case the separate continuity of the semigroup operation at the point $\alpha\cdot\beta$ in $\left(\mathscr{I}_\lambda^n(S),\tau_{\mathscr{I}}^\mathbf{c}\right)$ follows from cases  $(vi)$ and $(vii)$.

The previous statements of this section imply that $\tau_{\mathscr{I}}^\mathbf{c}\subseteq \tau_{\mathscr{I}}$ for any compact shift-continuous Hausdorff topology $\tau_{\mathscr{I}}$ on $\mathscr{I}_\lambda^n(S)$, and hence $\tau_{\mathscr{I}}^\mathbf{c}$ is the unique requested compact shift-continuous Hausdorff topology on $\mathscr{I}_\lambda^n(S)$.
\end{proof}

\begin{corollary}\label{corollary-5.15}
If $(S,\tau_S)$ is a compact Hausdorff semitopological inverse monoid with continuous inversion then $\left(\mathscr{I}_\lambda^n(S),\tau_{\mathscr{I}}^\mathbf{c}\right)$ is a compact Hausdorff semitopological inverse semigroup with continuous inversion.
\end{corollary}

\begin{proof}
Since $W_0^{-1}=\mathscr{I}_\lambda^n(S)\setminus\left({\Uparrow}S_{(a^1_1,\ldots,a^1_{p_1})}^{(b^1_1,\ldots,b^1_{p_1})}\cup \cdots\cup {\Uparrow}S_{(a^k_1,\ldots,a^k_{p_k})}^{(b^k_1,\ldots,b^k_{p_k})}\right)$ for an arbitrary basic neighbourhood $W_0=\mathscr{I}_\lambda^n(S)\setminus\left({\Uparrow}S^{(a^1_1,\ldots,a^1_{p_1})}_{(b^1_1,\ldots,b^1_{p_1})}\cup \cdots\cup {\Uparrow}S^{(a^k_1,\ldots,a^k_{p_k})}_{(b^k_1,\ldots,b^k_{p_k})}\right)$ of zero, inversion is continuous at zero in $\left(\mathscr{I}_\lambda^n(S),\tau_{\mathscr{I}}^\mathbf{c}\right)$.

Also, for an arbitrary element $\alpha=
\left(
\begin{array}{ccc}
a_1 & \ldots & a_k \\
s_1 & \ldots & s_k \\
b_1 & \ldots & b_k \\
\end{array}
\right)
$
of $\mathscr{I}_\lambda^n(S)$ and any its open neighbourhood $V_\alpha={\Uparrow}\left[V_1(s_1),\ldots,V_k(s_k)\right]^{(a_1,\ldots,a_k)}_{(b_1,\ldots,b_k)}\setminus \left({\Uparrow}S^{(a^1_1,\ldots,a^1_{l_1})}_{(b^1_1,\ldots,b^1_{l_1})}\cup\cdots\cup {\Uparrow}S^{(a^p_1,\ldots,a^p_{l_p})}_{(b^p_1,\ldots,b^p_{l_p})}\right)$ we have that $\left(V_\alpha\right)^{-1}\subseteq U_{\alpha^{-1}}$ for the neighbourhood $U_{\alpha^{-1}}={\Uparrow}\left[U_1(s^{-1}_1),\ldots,V_k(s^{-1}_k)\right]_{(a_1,\ldots,a_k)}^{(b_1,\ldots,b_k)}\setminus$ of $\alpha^{-1}$ in $\left(\mathscr{I}_\lambda^n(S),\tau_{\mathscr{I}}^\mathbf{c}\right)$ with
\begin{equation*}
  \left(V_1(s_1)\right)^{-1}\subseteq U_1(s^{-1}_1), \ldots , \left(V_k(s_k)\right)^{-1}\subseteq U_k(s^{-1}_k).
\end{equation*}
This completes the proof of the corollary.
\end{proof}

\end{document}